\documentclass{conm-p-l}
\usepackage{amssymb}

\def\done{}
\def\bbE{\mathrm{I\!E}}
\def\bbH{\mathrm{I\!H}}
\def\bbK{\mathrm{I\!K}}

\def\bbR{\mathrm{I\!R}}
\def\bbZ{\mathsf{Z\hskip-4.3ptZ}}
\def\rn{\bbR\nh^n}

\def\dx{\hs\dot{\nh x\hs}\nh}
\def\xq{L}

\def\bp{\mathbf{p}}
\def\dbp{\dot{\hskip0pt\bp\hskip0pt}}
\def\hbp{\hat{\mathbf{p}}}

\def\bx{\mathbf{x}}
\def\cat{\mathrm{cat}\hs}

\def\by{\mathbf{y}}
\def\dby{\dot{\hskip0pt\by\hskip0pt}}
\def\hby{\hat{\mathbf{y}}}
\def\ddby{\ddot{\hskip0pt\by\hskip0pt}}
\def\bz{b\hh}

\def\df{d\hskip-.8ptf}
\def\cj{c}
\def\m{m}
\def\p{p}
\def\q{i}
\def\wp{warp\-ed\hh-\hn prod\-uct}
\def\wp{warp\-ed\hh-\hn prod\-uct}
\def\hg{\hat{g\hskip2pt}\hskip-1.3pt}

\def\hyp{\hskip.5pt\vbox
{\hbox{\vrule width2.5ptheight0.5ptdepth0pt}\vskip2pt}\hskip.5pt}
\def\hs{\hskip.7pt}
\def\hh{\hskip.4pt}
\def\nh{\hskip-.7pt}
\def\nnh{\hskip-1.5pt}
\def\hn{\hskip-.4pt}
\def\w{^{\phantom i}}
\def\bM{\hskip3pt\overline{\hskip-3ptM\nh}\hs}
\def\bna{\hs\overline{\nh\nabla\nh}\hs}
\def\bg{\hskip1.2pt\overline{\hskip-1.2ptg\hskip-.3pt}\hskip.3pt}
\def\bM{\hskip3pt\overline{\hskip-3ptM\nh}\hs}
\def\br{\hskip3pt\overline{\hskip-3ptR\nh}\hs}
\def\bt{\kappa}
\def\cz{\gamma}
\def\tz{\delta}
\def\gz{\varepsilon}
\def\vg{\varGamma}
\def\ve{\varepsilon}
\def\h{\eta}
\def\lz{\mu}
\def\hi{\beta}
\def\eu{\chi}
\def\hps{\psi}
\def\vt{{\tau\hskip-4.55pt\iota\hskip.6pt}} 
\def\tv{\tilde{\tau\hskip-4.55pt\iota\hskip2pt}\hskip-1.4pt}
\def\area{\alpha}
\def\bvg{\hskip3pt\overline{\hskip-3pt\vg\nh}\hs}
\def\scal{\mathrm{s}}

\newtheorem{theorem}{Theorem}[section]
\newtheorem{lemma}[theorem]{Lemma}
\newtheorem{corollary}[theorem]{Corollary}
\newtheorem{qstn}[theorem]{Question}

\theoremstyle{definition}

\newtheorem{example}[theorem]{Example}

\theoremstyle{remark}
\newtheorem{remark}[theorem]{Remark}

\numberwithin{equation}{section}

\begin{document}

\title{Har\-mon\-ic-cur\-va\-ture warped products over surfaces}


\author[A. Derdzinski]{Andrzej Derdzinski}
\address{Department of Mathematics\\
The Ohio State University\hskip-1pt\\
231 W\hskip-2pt. 18th Avenue\\
Columbus, OH 43210, USA}
\email{andrzej@math.ohio-state.edu}
\thanks{Both authors' research was supported in part by a 
FAPESP\hn-\hh OSU 2015 Regular Research Award (FAPESP grant: 
2015/50265-6)}

\author[P. Piccione]{Paolo Piccione}
\address{Departamento de Matem\'atica\\
Instituto de Matem\'atica e Estat\'\i stica\\
Universidade de S\~ao Paulo\\
Rua do Mat\~ao 1010, CEP 05508-900\\
S\~ao Paulo, SP, Brazil}
\email{piccione@ime.usp.br}


\date{April 1, 2020}

\begin{abstract}
For warped products with harmonic curvature, nonconstant warping functions 
$\phi$, and compact two-dimensional bases $(M,h)$, we establish a dichotomy: 
either the Gaussian curvature $K$ of the metric $g=\phi^{-2}h$ is constant and 
negative, or $\phi$ equals a specific elementary function of $K$, also 
depending on the dimension $p$ and Einstein constant $\varepsilon$ of the 
fibre. In both cases the fibre must be an Einstein manifold with $p>1$ and 
$\varepsilon>0$, while the function $f=\phi^{p/2}$ satisfies a Yamabe-type 
second-order differential equation on $(M,g)$. We prove that both
possibilities are realized on every closed orientable surface of genus
greater than $1$, and in the latter case -- which also occurs on the
$2$-sphere and real projective plane -- the metrics in question constitute 
uncountably many distinct homothety types.
\end{abstract}

\maketitle


\setcounter{section}{0}
\section*{Introduction\done}
\setcounter{equation}{0}
One says that a Riemannian manifold with the curvature tensor $\,R\,$ has 
{\it harmonic curvature\/} \cite[Sect.~16.33]{besse} if 
$\,\mathrm{div}\,R=0\,$ or, in local coordinates, 
$\,R_{i\hn j\hh l}\w{}\hs^k{}\nnh_{,\hs k}\w=0$. This condition amounts to the
Co\-daz\-zi equation (\ref{cod}) imposed on the Ric\-ci tensor, and it
implies constancy of the scalar curvature \cite[Sect.~16.4(ii)]{besse}. A
compact Riemannian manifold has harmonic curvature if and only if its 
Le\-vi-Ci\-vi\-ta connection is a critical point of its Yang-Mills functional 
\cite[Sect.~16.34]{besse}.

The known examples of Riemannian manifolds with harmonic curvature include 
five non-disjoint classes, consisting of: {\it Ein\-stein manifolds, 
con\-for\-mal\-ly flat manifolds of constant scalar curvature, locally 
reducible manifolds having\/} $\,\mathrm{div}\,R=0$, {\it certain nontrivial
warped products of dimensions\/} $\,n>4\,$ {\it with Ein\-stein fibres and 
one\hs-di\-men\-sion\-al or hyperbolic bases, and some four-man\-i\-folds that 
are, locally or globally, nontrivial warped products of surfaces}. See
\cite[Sect.~16.34,~16.40]{besse}, \cite[Sect.~4]{derdzinski-88}. Every known
{\it compact\/} example belongs to one of these five classes.

In the construction of the last two classes of the preceding paragraph, 
except for the case of hyperbolic bases, the (local) warp\-ed-\hn prod\-uct 
structure, rather than being an Ansatz, follows from a purely geometric 
assumption, namely, that a certain tensor field $\,B\,$ should have no more 
than two distinct eigen\-val\-ues at each point. Specifically, in the fourth 
(or, fifth) class, $\,B\,$ is the Ric\-ci tensor \cite[Sect.~16.38]{besse} 
or, respectively, the self-dual Weyl tensor acting on self-dual bi\-vec\-tors 
\cite[p.\ 145]{derdzinski-88}. Warped products with harmonic curvature were
studied by Kim, Cho and Hwang \cite{kim-cho-hwang}, and there are interesting 
results on Ein\-stein warped products \cite{he-petersen-wylie,lu-page-pope}.

It is therefore natural to consider the following problem.
\begin{qstn}\label{which}
Which compact warp\-ed-\hn prod\-uct Riemannian manifolds of dimensions 
greater than\/ $\,3\,$ 
have harmonic curvature, without belonging to the first three 
classes italicized above?
\end{qstn}
Question~\ref{which} remains open in general, and its complexity clearly 
increases with the dimension $\,\m\,$ of the base. For $\,\m=1\,$ the answer
is well known \cite[Lemma 1(ii) and Theorem 1]{derdzinski-82}. The present
paper deals with the case $\,\m=2$.

We begin by proving a dichotomy result (Theorem~\ref{srfbs}): if a warped 
product has harmonic curvature, a com\-pact-sur\-face base $\,(M\nh,h)$, and 
a nonconstant warping function $\,\phi$, while $\,K\hs$ denotes the 
Gauss\-i\-an curvature of the con\-for\-mal\-ly-re\-lat\-ed metric 
$\,g=\phi^{-\nh2}h\,$ on $\,M\nh$, then
\begin{equation}\label{fbm}
\begin{array}{l}
\mathrm{the\hs\ fibre\hs\ must\hs\ be\hs\ an\hs\ Ein\-stein\hs\ manifold\hs\
of\hs\ some}\\
\mathrm{dimension\ }\hs\p\hn>\nh1\mathrm{\ with\ an\ Ein\-stein\ constant\ 
}\gz>0
\end{array}
\end{equation}
(for compact bases of all dimensions; see Remark~\ref{parts}), the function
$\,f\nh=\phi^{\,p/2}$ satisfies a Ya\-ma\-be-type sec\-ond-or\-der
differential equation (\ref{ktr}.iii) on $\,(M\nh,g)$, and
\begin{equation}\label{dch}
\begin{array}{l}
\mathrm{either\ }\,\,K\hs\,\mathrm{\ is\ the\ negative\ constant\
}\,-\nnh\gz/\nh(\p\hh-\nh1)\mathrm{,\ or\ }\,\,\phi\,\,\mathrm{\ equals\ a\
positive}\\
\mathrm{constant\ times\
}\,|\hh(\p\hh-\nh1)K\nnh+\gz\hs|^{1/(1\nh-\p)}\nh\mathrm{,\
with\ }\nh\,K\hs\mathrm{\ (necessarily)\ nonconstant.}
\end{array}
\end{equation}
Conversely, these conditions imply harmonic curvature for the warped product.

Theorem~\ref{srfbs} is a the first step toward answering Question~\ref{which} 
for $\,m=2$, and the two cases of (\ref{dch}) amount to two very different
problems.

The first one concerns finding, for $\,\p,\gz\,$ fixed as in (\ref{fbm}),
{\it nonconstant positive solutions\/ $\,f\hs$ to the quasi\-lin\-e\-ar
elliptic equation}
\begin{equation}\label{qle}
\Delta\nh f\nh-a\nh f\nh=-\hn cf\hh^{1+4/\p}\,\hs\mathrm{\ with\ constants\ 
}\,a=\p(\p\hh-\nh2)\gz/[4(\p\hh-\nh1)]\,\mathrm{\ and\ }\,c>0
\end{equation}
{\it on a given closed surface of negative constant
Gauss\-i\-an curvature\/} $\,K\nnh=-\gz/\nh(\p\hh-\nh1)$. (This is equation
(\ref{ktr}.iii), with $\,r=0<c\,$ due to (\ref{css}) --
(\ref{bth}.i).) Ya\-ma\-be \cite{yamabe} has shown -- cf.\ Lemma~\ref{ymabe}
below -- that such $\,f\,$ exist, on any compact Riemannian surface, if the
parameters $\,a,\p\,$ satisfy a specific inequality, which here reads
\begin{equation}\label{pgt}
\p\,>\,2\,-\,\lambda_1\w/K.
\end{equation}
As we point out in Remark~\ref{eighp}, a result of Schoen, Wolpert and Yau
\cite{schoen-wolpert-yau} yields
\begin{equation}\label{tgl}
\lambda_1\w\hs<\,2|K\hn|
\end{equation}
whenever the metric of constant curvature $\,K\nnh<0$, on any closed
orientable surface of genus $\,\mathbf{g}>1$, is confined to a suitable 
nonempty open subset of the Teich\-m\"ul\-ler space; (\ref{tgl}) gives 
(\ref{pgt}) for all $\,\p\ge4$, and the fibre dimensions $\,\p\ge4\,$ are the
only ones of interest for the ``constant $\,K\hs$ case'' of
Question~\ref{which} (see Remark~\ref{cnffl}).

Consequently, {\it the first possibility in\/} (\ref{dch}) {\it is realized,
with warped products of all relevant fibre dimensions, by a 
Teich\-m\"ul\-ler-open nonempty set of metrics of constant curvatures\/}
$\,K\nnh<0$, {\it on closed orientable surfaces of all genera\/}
$\,\mathbf{g}>1$.

In the remaining, second case of (\ref{dch}) we look for {\it metrics\/}
$\,g\,$ {\it on compact surfaces\/} $\,M\,$ {\it having nonconstant
Gauss\-i\-an curvatures\/} $\,K\hs$ {\it such that there exist positive
constants\/} $\,\gz,\lz\in\bbR\,$ {\it with\/} 
$\,(\p\hh-\nh1)K\nnh+\gz\ne0\,$ {\it everywhere in\/} $\,M\,$ {\it and}
\begin{equation}\label{ngc}
\begin{array}{l}
2(\p+\hn1)\hs[(\p\hh-\nh1)K\nnh+\gz]\hh\Delta\hn K\,
-\,(3\p\hh-\nh2)(\p+\hn1)\hs g(\nabla\nnh K\nh,\nnh\nabla\nnh K)]\\
\hskip48pt=\,\lz\hs|(\p\hh-\nh1)K\nnh+\gz\hs|^{2(\p\hh-\nh2)/(\p\hh-\nh1)}\hs
-\,(2K\nnh+\p\hs\gz)\hh[(\p\hh-\nh1)K\nnh+\gz]^2.
\end{array}
\end{equation}
(Equation (\ref{ngc}), that is, (\ref{csb}), requires a normalization of
the warping function, described in Section~\ref{wt}.) Let us emphasize that
the existence of $\,\gz,\lz\in(0,\infty)\,$ for which (\ref{ngc}) holds and 
$\,|(\p\hh-\nh1)K\nnh+\gz\hs|>0\,$ on $\,M\,$ is {\it a property of the
metric\/} $\,g\,$ {\it alone}. Using a bifurcation argument, we prove,
in Section~\ref{cn}, that metrics $\,g\,$ with this property exist for $\,M\,$
dif\-feo\-mor\-phic to $\,S^2\nh,\,\bbR\mathrm{P}^2$ or any closed orientable
surface of genus greater than $\,1$. More precisely, curves of such metrics,
emanating from a given metric $\,\hg\,$ of (nonzero) constant Gauss\-i\-an
curvature $\,\hat K\hs$ on $\,M\nh$, are naturally associated with certain
positive eigen\-val\-ues $\,\lambda\,$ of $\,-\nh\hat\Delta$, for the
La\-plac\-i\-an $\,\hat\Delta\,$ of $\,\hg\nh$. Each of the curves in question,
which we call $\,\lambda${\it-branches}, consists of metrics representing
uncountably many distinct homothety types and, if 
$\,\lambda\nh'\hn\ne\lambda$, a metric from the $\,\lambda$-branch, close to 
$\,\hg\nh$, cannot be homo\-thet\-ic to any metric near $\,\hg\,$ belonging to 
the $\,\lambda\nh'\nh$-branch. Here are some further details.

If $\,\hat K\nnh>0$, the eigen\-val\-ues $\,\lambda>0\,$ that give rise to
$\,\lambda$-branches may be completely arbitrary (on $\,\bbR\mathrm{P}^2$), or
even-num\-ber\-ed and otherwise arbitrary (on $\,S^2$). For $\,\hat K\nnh<0$
(that is, on any closed orientable surface of genus greater than $\,1$)
these $\,\lambda\,$ have to be simple and different from 
$\,(\p\hh-\nh2)\hh|\hat K\hn|$, and so, according to the theorem of Schoen,
Wolpert and Yau \cite{schoen-wolpert-yau} mentioned in Remark~\ref{eighp},
con\-stant-cur\-va\-ture metrics $\,\hg$ admitting such eigen\-val\-ues
$\,\lambda\,$ fill a nonempty open subset of the Teich\-m\"ul\-ler space.

As a result, {\it warped products of all fibre dimensions\/} $\,\p>1\,$
{\it realize the second case of\/} (\ref{dch}) {\it with\/} $\,M\nh=S^2\nh$,
{\it or\/} $\,M\nh=\bbR\mathrm{P}^2\nh$, {\it or\/} $\,M\,$ {\it closed,
orientable and of any genus greater than\/} $\,1$, {\it while -- in the last 
case -- the con\-for\-mal types of the metrics\/} $\,g\hs\,$ {\it form a
nonempty Teich\-m\"ul\-ler-open set}.

Two subcases of the second case of (\ref{dch}) need commenting on. One,
characterized by $\,\p=2$, has already been settled in \cite{derdzinski-88}.
The other, in which $\,M\nh=T^2\nh$, is still an open problem, even though one
can easily provide examples of nontrivial compact warped products with
harmonic curvature and bases dif\-feo\-mor\-phic to $\,T^2$ that are neither
Ein\-stein nor con\-for\-mal\-ly flat: namely, Riemannian products of $\,S^1$
and suitably chosen har\-mon\-ic-cur\-va\-ture warp\-ed-\hn prod\-uct
manifolds having the base $\,S^1\nh$, classified in \cite[Lemma 1(ii) and
Theorem 1]{derdzinski-82}. However, being reducible, such examples do not lie
within the scope of Question~\ref{which}.

\section{Notations and preliminaries\done}\label{pr}
\setcounter{equation}{0}
Manifolds (always assumed connected), mappings and tensor fields, including 
Riemannian metrics and functions, are by definition of class $\,C^\infty\nnh$, 
except in Sections~\ref{nc} --~\ref{cb} where, for technical reasons, we 
require that, rather than being smooth, they should belong to suitable 
$\,L\nh^2$ So\-bo\-lev spaces. Given a Riemannian metric $\,g$, we let 
$\,\nabla$ stand for the Le\-vi-Ci\-vi\-ta connection of $\,g\,$ as well as 
the $\,g$-grad\-i\-ent, and $\,\Delta,\mathrm{Ric},\mathrm{div},K\,$ for the 
$\,g\hh$-\nh La\-plac\-i\-an, the Ric\-ci tensor of $\,g$, the 
$\,g$-di\-ver\-gence and, in the case of a surface metric $\,g$, its 
Gauss\-i\-an curvature. When a metric is denoted by $\,h$, the analogous
symbols will be $\,D\nh,\Delta\nnh^h\nnh,\mathrm{Ric}^h\nnh,\mathrm{div}^h$
and $\,K\hn^h\nh$.

One calls a function $\,\hi$ on Riemannian manifold $\,(M\nh,g)\,$ {\it 
iso\-par\-a\-met\-ric\/} \cite{wang} if $\,\Delta\nh\hi\,$ and 
$\,g(\nabla\nh\hi,\nnh\nabla\nh\hi)\,$ are functions of $\,\hi$. It is
well-known that, when $\,\dim M=2$, the existence of nonconstant
iso\-par\-a\-met\-ric functions amounts (locally, at generic points) to
``\hs rotational symmetry.'' More precisely, for $\,\hi:M\to\bbR\,$ and a
Kil\-ling field $\,v\,$ with $\,d_v\w\hi=0\,$ and $\,\nabla\nh\hi\ne0\ne v\,$
everywhere, $\,\hi\,$ must be iso\-par\-a\-met\-ric since the flow of $\,v\,$
leaves $\,\Delta\hi\,$ and $\,g(\nabla\nh\hi,\nnh\nabla\nh\hi)\,$ invariant
as well. Conversely, on an oriented Riemannian surface $\,(M\nh,g)$,
iso\-par\-a\-met\-ric\-i\-ty of a function $\,\hi$ without critical points
leads to an explicit construction of a Kil\-ling field $\,v\,$ without zeros,
orthogonal to $\,\nabla\nh\hi$. See, e.g., \cite[Lemma 7]{derdzinski-88}, or
formula (\ref{kil}.ii) below.

Here is another well-known fact, cf.\
\cite[Remark 2.5]{derdzinski-piccione-20} or the end of this section:
\begin{equation}\label{ext}
\begin{array}{l}
\mathrm{any\ Kil\-ling\nh\ vector\ field\ defined\ on\ a\ nonempty\nh\
connected\ open}\\
\mathrm{subset\ of\ a\ simply\ connected,\nnh\
real\hh}\hyp\mathrm{an\-a\-lyt\-ic,\nh\nnh\ Riemannian\ mani}\hyp\\
\mathrm{fold\ }\,(M\nh,h)\mathrm{,\ can\ be\ uniquely\ extended\ to\ a\
 Kil\-ling\ field\ on\ }\,M\nh.
\end{array}
\end{equation}
\begin{lemma}\label{glbki}Let a compact real-an\-a\-lyt\-ic Riemannian 
surface\/ $\,(M\nh,h)\,$ have nonconstant Gauss\-i\-an curvature\/ $\,K\hn^h$ 
and nonzero Euler characteristic\/ $\,\eu(M)$. Any\/ $\,h$-Kil\-ling vector 
field\/ $\,v$ defined on a nonempty connected open set\/
$\,\,U\subseteq M\,$ has a unique extension to an\/ $\,h$-Kil\-ling field on\/
$\,M\nh$, provided that, if necessary, one replaces\/ $\,(M\nh,h)$ by a
two\hh-\hn fold isometric covering thereof, and\/ $\,\,U\hn\,$ by a connected
component of the pre-im\-age of\/ $\,\,U\hn\,$ under the covering projection.
\end{lemma}
\begin{proof}Assuming $\,v\,$ to be nontrivial and denoting by
$\,(M\hn'\nnh,h')\,$ the Riemannian universal covering of $\,(M\nh,h)$, we
see that $\,v\,$ gives rise to a $\,h'\nh$-Kil\-ling field $\,v'$ on a
suitable (connected) open sub\-man\-i\-fold $\,\,U'$ of $\,M\nh$, and 
(\ref{ext}) allows us to treat $\,v'$ as defined on all of $\,M'\nh$. Then, 
push-for\-wards of $\,v'$ under deck transformations are constant multiples of
$\,v'$ (or else $\,K\hn^h$ would be constant), and $\,v'$ has zeros (or else it
would span a tangent-direction field on $\,M'\nh$, descending to $\,M\nh$,
even though $\,\eu(M)\ne0$). As the flow of $\,v'$ is periodic due to its
obvious periodicity on a neighborhood of a zero of $\,v'\nh$, the
push-for\-wards of $\,v'$ under deck transformations, having the same flow
period as $\,v'$ itself, must all equal $\,\pm v'\nh$.
\end{proof}
For any function $\,\hi\,$ on a Riemannian manifold $\,(M\nh,g)\,$ one
clearly has
\begin{equation}\label{tnd}
2\hh[\nabla\nh d\hi](v,\,\cdot\,)=dQ\hh,\hskip10pt\mathrm{where}\hskip5pt
v=\nabla\nh\hi\hskip5pt\mathrm{and}\hskip5ptQ
=g(\nabla\nh\hi,\nnh\nabla\nh\hi)\hh.
\end{equation}
Suppose now that we are given a $\,1$-form $\,\xi\,$ and a 
twice\hh-co\-var\-i\-ant symmetric tensor field $\,\bz\,$ on a Riemannian 
manifold $\,(M\nh,h)$. Treating $\,\bz\,$ as a $\,1$-form valued in
$\,1$-forms, we define the exterior product $\,\xi\wedge\bz\,$ and the
exterior derivative $\,d\bz\,$ to be the $\,2$-forms valued in $\,1$-forms
with the lo\-cal-co\-or\-di\-nate expressions 
$\,[\xi\wedge\bz]_{qr\nh s}\w=\xi_q\w\bz_{r\hn s}\w-\xi_r\w\bz_{qs}\w$ 
and $\,[d\bz]_{qr\nh s}\w=\bz_{r\hn s,q}\w-\bz_{qs,r}\w$ or, in 
co\-or\-di\-nate-free notation, 
$\,[\xi\wedge\bz](u,v)=\xi(u)\hs\bz(v,\,\cdot\,)-\xi(v)\hs\bz(u,\,\cdot\,)\,$
and $\,[d\bz](u,v)=(D\hskip-1pt_u\w\bz)(v,\,\cdot\,)
-(D\hskip-1pt_v\w\bz)(u,\,\cdot\,)\,$ for tangent vector fields $\,u,v\,$ and 
the Le\-vi-Ci\-vi\-ta connection $\,D\,$ of $\,h$. Then, cf.\
\cite[Sect.~16.3]{besse},
\begin{equation}\label{cod}
d\bz=0\,\,\mathrm{\ if\ and\ only\ if\ }\,\bz\,\mathrm{\ is\ a\ Co\-daz\-zi\ 
tensor\ on\ }(M\nh,h)
\end{equation}
while, for any functions $\,f,\phi:M\to\bbR$, with $\,\dim M=2\,$ in
(\ref{dfa}.b) -- (\ref{dfa}.c),
\begin{equation}\label{dfa}
\begin{array}{rl}
\mathrm{a)}&d[f\nh\bz]\,\hs=\,f\hn d\bz\,\hh+\,d\hskip-1.4ptf\nnh\wedge\bz\hh,
\hskip14pt\mathrm{b)}\hskip6ptd[Dd\phi]\,=\,-\hn K\hn^h\nh d\phi\wedge h\hh,\\
\mathrm{c)}&K\hn^h\nh d\phi\,\hh=\,\mathrm{div}^h\nh[Dd\phi]\,
-\hs\,d\Delta\nnh^h\nnh\phi.
\end{array}
\end{equation}
Namely, (\ref{dfa}.b) amounts to the Ric\-ci identity
for $\,d\phi\,$ expressed in terms of the Gauss\-i\-an curvature
$\,K\hn^h\nh$, that is, the two\hs-di\-men\-sion\-al case of the general
formula $\,d\nabla\nh\xi=\xi R\,$ (in coordinates: 
$\,\xi_{s,jq}\w\hs-\,\xi_{s,qj}\w\hs=\,R\nh_{qjs}\w{}^p\xi_p\w$), applied here
to $\,\xi=d\phi$, and valid for any $\,1$-form $\,\xi\,$ on a manifold with a
tor\-sion-free connection $\,\nabla$ having the curvature tensor $\,R$. (The
exterior derivative $\,d\bz\,$ of $\,\bz=\nabla\nh\xi\,$ is defined, as above,
by $\,[d\bz]_{qr\nh s}\w\nh=\bz_{r\hn s,q}\w\nh-\hs\bz_{qs,r}\w$, but this
time the twice\hh-co\-var\-i\-ant tensor field $\,\bz\,$ need not be
symmetric.) Contracting (\ref{dfa}.b), one gets the Boch\-ner identity
(\ref{dfa}.c).
\begin{lemma}\label{extpr}
Let\/ $\,J\,$ and\/ $\,\area\,$ be the com\-plex-struc\-ture tensor and the 
area\/ $\,2$-form of an oriented two\hs-di\-men\-sion\-al Riemannian manifold 
$\,(M\nh,h)$, with the convention that\/ $\,\area_{qr}\w=J_q^sh_{sr}\w$ or,
equivalently, $\,\area=h(J\hs\cdot\,,\,\cdot\,)$. Any\/ $\,1$-form\/
$\,\xi\hn$ and twice\hh-co\-var\-i\-ant symmetric tensor field\/ $\hs\bz\hn$
on\/ $\,M\,$ then satisfy the relation
\[
J^*\nnh(\xi\wedge\bz)\,
=\,\area\hs\otimes[(\mathrm{tr}_h\w\bz)\hs\xi-\bz(v,\,\cdot\,)]\hh,\mathrm{\ \ 
for\ }\,v\,\mathrm{\ characterized\ by\ }\,\xi=h(v,\,\cdot\,)\hh,
\]
the lo\-cal-co\-or\-di\-nate version of which reads\/ 
$\,(\xi_q\w\bz_{r\hn s}\w-\xi_r\w\bz_{qs}\w)J_{\nh i}^s
=(\bz_{\!s}^s\xi_i\w-\bz_i^s\xi_s\w)\hs\area_{qr}\w$. 
Two special cases arise when\/ $\,\bz=h\,$ or, respectively, $\,\xi=d\phi\,$ 
and\/ $\,\bz=Dd\phi\,$ with a function\/ $\,\phi:M\to\bbR$. Namely, if\/
$\,D\,$ denotes both the Le\-vi-Ci\-vi\-ta connection of\/ $\,(M\nh,h)\,$ and
the\/ $\,h$-grad\-i\-ent, $\,\Delta\nnh^h$ is the $\,h$-La\-plac\-i\-an, 
and\/ $\,H=h(D\phi,D\phi):M\to\bbR$,
\begin{equation}\label{jxh}
J^*\nnh(\xi\wedge h)\,=\,\area\hs\otimes\hs\xi\hh,\hskip16pt
J^*\nnh(d\phi\wedge\nh Dd\phi)\,
=\,\area\hs\otimes[(\Delta\nnh^h\nnh\phi)\hs d\phi-dH\nnh/2]\hh.
\end{equation}
\end{lemma}
\begin{proof}Being skew-sym\-me\-tric in $\,q,r$, the expression 
$\,\xi_q\w\bz_{r\hn s}\w -\xi_r\w\bz_{qs}\w$ must equal
$\,\rho_s\w\area_{qr}\w$, for some $\,1$-form $\,\rho$. Contracting this 
equality against $\,h^{r\hn s}$ we obtain 
$\,\bz_{\!s}^s\xi_i\w-\bz_i^s\xi_s\w=\area_{ir}\w w^r\nh
=J_i^sh_{sr}\w w^r\nh=J_i^s\rho_s\w$, with the vector field $\,w\,$ given by 
$\,\rho=h(w,\,\cdot\,)$, which proves our formula for
$\,J^*\nnh(\xi\wedge\bz)$, and (\ref{jxh}) now follows from (\ref{tnd}).
\end{proof}
Given a Riemannian surface $\,(M\nh,g)\,$ and 
$\,\hi,\sigma,\zeta:M\to\bbR\,$ such that $\,\sigma\,$ and $\,\zeta$ are
functions of $\,\hi$, while $\,\Delta\hi=\sigma\,$ and 
$\,g(\nabla\nh\hi,\nnh\nabla\nh\hi)=2\hh\zeta$, one has
\begin{equation}\label{kil}
\begin{array}{rl}
\mathrm{i)}&
2\zeta\hs\nabla\nh d\hi\,
=\,2(\sigma-\zeta')\zeta g\,
+\,(2\zeta'\nh-\sigma)\,d\hi\nh\otimes\nh d\hi\hh,\hskip9pt\mathrm{with\
}\,\zeta'\nh=d\hs\zeta/d\hi\hh,\\
\mathrm{ii)}&
\mathrm{if\ the\ surface\ }\,(M\nh,g)\,\mathrm{\ is\ oriented,\ then\ 
}\,e^\bt\hskip-2.7ptJv\hskip7pt\mathrm{is\ a\
}\,g\hh\hyp\mathrm{Kil\-ling\ field,}
\end{array}
\end{equation}
for $\,v=\nabla\nh\hi\,$ and $\,J\,$ as in Lemma~\ref{extpr}. Here
(\ref{kil}.ii) holds on the open set $\,\,U\,$ on which $\,d\hi\ne0$, the
function $\,\bt\,$ of $\,\hi\,$ being any antiderivative of
$\,(\sigma-2\zeta')/(2\zeta)$, defined away from zeros of $\,\zeta$. Namely,
both sides of (\ref{kil}.i) are symmetric, have the same $\,g$-trace, and
agree, when evaluated on $\,\nabla\nh\hi$, as a consequence of (\ref{tnd}),
which yields (\ref{kil}.i) both on our open set $\,\,U\nnh$, and in the
interior of the zero set of $\,d\hi$, while the union of the two sets is
dense. To obtain (\ref{kil}.ii), note that
$\,2\zeta\hs e^{-\bt}\hh\nabla(e^\bt\hskip-2.7ptJv)=2(\sigma-\zeta')\zeta J\,$
in view of the relations 
$\,2\zeta\hs d\bt\otimes Jv=(\sigma-2\zeta')\,d\hi\otimes Jv$ and
$\,2\zeta\hs J\nabla\nh v
=2(\sigma-\zeta')\zeta J+(2\zeta'\nh-\sigma)\,d\hi\nh\otimes Jv$ due,
respectively, to our choice of $\,\bt$, and to (\ref{kil}.i) rewritten as
$\,2\zeta\hs\nabla\nh v=2(\sigma-\zeta')\zeta
+(2\zeta'\nh-\sigma)\,d\hi\nh\otimes v$, where $\,2(\sigma-\zeta')\zeta$
stands for $\,2(\sigma-\zeta')\zeta\,$ times the identity.
\begin{lemma}\label{nabdk}
For a Riemannian surface\/ $\hs(M\nh,g)\hs$ with the Gauss\-i\-an 
\hbox{curvature\/ $\hs K\nh$,}
\begin{enumerate}
  \def\theenumi{{\rm\roman{enumi}}}
\item whenever\/ $\,\psi,\nu:M\to\bbR\,$ are functions with\/ 
$\,\nabla\nh dK=\psi g+\nu\,dK\nnh\otimes\hh dK\nh$, one necessarily has\/ 
$\,(K\nh-\psi\nu)\,dK\nh+\hs d\psi
=g(\nabla\nnh K\nh,\nnh\nabla\nh\nu)\,dK\nh
-\hs g(\nabla\nnh K\nh,\nnh\nabla\nnh K)\,d\nu$,
\item if functions\/ $\,\varSigma,Z\,$ defined on an interval containing
the range of\/ $\,K\hs$ satisfy the relations\/ $\,\Delta\hn K=\varSigma(K)\,$ 
and\/ $\,g(\nabla\nnh K,\nnh\nabla\nnh K)=2Z(K)$, then
\begin{equation}\label{zpm}
(2Z'\nh-\varSigma)(Z'\nh-\varSigma)
=2(Z''\nh-\hs\varSigma'-\hs K)\hh Z\hh,\hskip12pt\mathrm{where}
\hskip7pt(\hskip2.3pt)'\nh=d/dK\nh.
\end{equation}
\end{enumerate}
\end{lemma}
\begin{proof}In (i), $\,\Delta\hn K=2\psi
+\nu\hs g(\nabla\nnh K,\nnh\nabla\nnh K)$, and our claim immediately follows
from (\ref{dfa}.c) for $\,h=g\,$ and $\,\phi=K\nh$, combined with (\ref{tnd}).
Under the hypotheses of (ii), we may apply (\ref{kil}.i), at points with
$\,d\hi\ne0$, to $\,(\hi,\sigma,\zeta)=(K\nh,\varSigma,Z)$, obtaining the
assumption of (i) for $\,\psi=\varSigma-Z'$ and $\,\nu=(Z'\nh-\varSigma/2)/Z$.
The conclusion of (i) now reads $\,(K\nh-\psi\nu)\,dK\nh+\hs d\psi=0$, since
$\,\nu\,$ is a function of $\,K\nh$, and it easily gives (\ref{zpm}) wherever
$\,Z\ne0$. Thus, (\ref{zpm}) holds on the whole interval in question, due to
a dense\hh-union argument analogous to the one following (\ref{kil}); note
that $\,\varSigma=0\,$ on every open interval on which $\,Z=0$.
\end{proof}
\begin{remark}\label{sccrv}The scalar curvatures
$\,\scal,\,\scal^h$ and La\-plac\-i\-ans
$\,\Delta,\Delta\nnh^h$ of con\-for\-mal\-ly related Riemannian metrics 
$\,g\,$ and $\,h=g/\hn\vt^2$ in dimension $\,\m\,$ are given by
\begin{equation}\label{sca}
\scal^h\nh=\vt^2\scal+2(\m-1)\vt\Delta\vt
-\m(\m-1)g(v,v)\hh,\hskip14pt\Delta\nnh^h\nh
=\vt^2\nh\Delta-(\m-2)\vt\hs d_v\w\hh, 
\end{equation}
where $\,v=\nabla\hskip-1.1pt\vt\,$ is the $\,g$-grad\-i\-ent of $\,\vt$. Cf.\ 
\cite[Theorem 1.159]{besse}. For $\,\m=2$, this becomes 
$\,K\hn^h\nh=\vt^2\nh K\nnh+\vt\Delta\vt-g(v,v)\,$ and 
$\,\Delta\nnh^h\nh=\vt^2\nh\Delta$, with $\,\scal=2K\,$ and 
$\,\scal^h\nh=2K\hn^h$ expressed in terms of the Gauss\-i\-an 
curvatures $\,K,K\hn^h$ of $\,g\,$ and $\,h$.
\end{remark}
\begin{remark}\label{ymuni}Under the assumptions of Remark~\ref{sccrv},
if $\,\vt\,$ assumes its extremum values
$\,\vt\nh_{\mathrm{max}}\w,\,\vt\nh_{\mathrm{min}}\w$, while 
$\,\scal,\,\scal^h$ are both constant and 
$\,\scal<0<\vt$, then $\hs\vt\hs$ is constant and $\,\scal^h<0$. 
This well-known conclusion follows since, by (\ref{sca}), 
$\,\vt^2_{\hskip-.8pt\mathrm{max}}
\le\hs\scal^h\hskip-1.7pt/\scal
\le\hs\vt^2_{\hskip-.8pt\mathrm{min}}$.
\end{remark}
\begin{remark}\label{holom}The existence of a nontrivial $\,h$-Kil\-ling 
vector field $\,v$, for a Riemannian metric $\,h\,$ on a compact surface
$\,M\nh$, precludes negativity of the Gauss\-i\-an curvature $\,K\hs$ of {\it
any\/} metric $\,g\,$ on $\,M\nh$. In fact, passing to a two\hh-\hn fold
covering, if necessary, we may assume $\,M\,$ oriented, which turns $\,h\,$
into a K\"ah\-ler metric on a compact complex curve of some genus
$\,\mathbf{g}$, admitting a nontrivial real-hol\-o\-mor\-phic vector field
$\,v\,$ (so that $\,\mathbf{g}\le1$), while the condition $\,K\nnh<0\,$ would
give $\,\mathbf{g}>1$.
\end{remark}
\begin{remark}\label{cnfpr}A Riemannian product with factors of dimensions
$\,n\,$ and $\,n'$ is con\-for\-mal\-ly flat if and only if both factors have 
constant sectional curvatures $\,K\nh,K'$ and
$\,(n-1)(n'\nh-1)(K\nh+K')=0$. See \cite[Section 5]{yau}, as well as 
\cite[subsection 1.167]{besse}.
\end{remark}
\begin{remark}\label{dlfeq}Let $\,\Delta\nh f\nh=\varOmega(f)\,$ for a 
function $\,f\,$ on a compact Riemannian manifold and a function 
$\,\varOmega\,$ on an interval containing the range of $\,f\nh$. If 
$\,\varOmega\,$ is strictly increasing or constant, then $\,f\,$ must be 
constant, since $\,\varOmega(f\nnh_{\mathrm{max}}\w\nh)\le0
\le\varOmega(f\nnh_{\mathrm{min}}\w\hn)$.
\end{remark}
\begin{remark}\label{eignv}Let $\,\lambda_j\w$ be the $\,j\hh$th
eigen\-val\-ue of $\,-\nh\hat\Delta$, for the La\-plac\-i\-an $\,\hat\Delta\,$
of the $\,2$-sphere (or, projective plane) of constant Gauss\-i\-an curvature 
$\,\hat K\nh$, with
\begin{equation}\label{eig}
\lambda_0\w\,<\hs\lambda_1\w\,<\hs\lambda_2\w\,<\hs\ldots\hs.
\end{equation}
Then $\,\lambda_j\w=j(j+1)\hat K\,$ (or, $\,\lambda_j\w=2j(2j+1)\hat K$). The
spectrum of  $\,-\nh\hat\Delta\,$ acting on rotationally invariant functions
is the same, but with one\hs-di\-men\-sion\-al eigen\-spaces, spanned by the
{\it zonal spherical harmonics\/} \cite[Sect.\ 2.3]{zelditch}.
\end{remark}
\begin{remark}\label{eighp}On every closed orientable surface $\,M\,$ of genus 
greater than $\,1\,$ there exist metrics $\,g\,$ having constant Gauss\-i\-an 
curvature $\,K\nnh<0\,$ and an arbitrarily large number of eigen\-values of
$\,-\nh\Delta\,$ in $\,(|K\hn|/4,t+|K\hn|/4\hh]$, for any 
$\,t\in(0,\infty)$, where $\,\Delta\,$ is the La\-plac\-i\-an; all such
metrics obviously satisfy (\ref{tgl}). See \cite[p. 211, Theorem 8.12]{buser},
\cite[p. 251, Theorem 2]{chavel}. Also, $\,M\,$ then admits a
sequence of metrics $\,g\,$ with $\,K\nh=-\nh1\,$ for which the lowest
positive eigen\-value $\,\lambda_1\w$ of $\,-\nh\Delta\,$ tends to $\,0\,$ and
has multiplicity $\,1$. This follows from a result of Schoen, Wolpert and Yau 
\cite[three final lines of the first paragraph on p.\
279]{schoen-wolpert-yau}, applied to $\,n=1$.
\end{remark}
Finally, (\ref{ext}) follows since one can treat Kil\-ling fields as sections
of a certain vector bundle, parallel relative to a natural connection. This is
why the same conclusion holds, more generally, both for con\-for\-mal vector
fields in the pseu\-\hbox{do\hs-}Riem\-ann\-i\-an case, and for infinitesimal
affine transformations on a manifold with a connection, cf.\
\cite[lines following Lemma 9.1]{derdzinski-11}, 
\cite[text surrounding formula (1.5)]{derdzinski-piccione-17}.

\section{Warped products and harmonic curvature\done}\label{wp}
\setcounter{equation}{0}
Given Riemannian manifolds $\,(M\nh,h),\,(\varPi\nh,\h)\,$ of positive 
dimensions $\,m,\p\,$ and a nonconstant function $\,\phi:M\to(0,\infty)$, 
consider the nontrivial warped product
\begin{equation}\label{nwp}
(M\times\hh\varPi\nh,\,h+\nh\phi^2\h)
\end{equation}
with the base $\,(M\nh,h)$, fibre $\,(\varPi\nh,\h)\,$ and warping function 
$\,\phi$. (The word `nontrivial' refers to nonconstancy of $\,\phi$, and the 
same symbols $\,h,\h,\phi\,$ represent the pull\-backs of $\,h,\h,\phi\,$ to 
the product $\,M\times\hh\varPi$.) The \wp\ metric of (\ref{nwp}) is obviously 
con\-for\-mal to a product metric: 
$\,h+\nh\phi^2\h=\phi^2\hh[\hs g+\h\hs]$, where $\,g=\phi^{-\nh2}\nh h$.
As one easily verifies \cite[Lemma 1.2]{derdzinski-piccione-20}, cf.\ also 
\cite{kim-cho-hwang}, the (nontrivial) warped product (\ref{nwp}) has harmonic 
curvature if and only if there exists a constant $\,\gz\in\bbR\,$ such that
\begin{enumerate}
  \def\theenumi{{\rm\roman{enumi}}}
\item $\mathrm{Ric}^h\nh-\hh\p\hh\phi^{-\nnh1}Dd\phi\,$ is a 
Co\-daz\-zi tensor on $\,(M\nh,h)$,
\item $\mathrm{div}^h(\phi\hs^{\p\hn-\nh2}Dd\phi)
=[(\p\hh-\nh1)\varLambda-\gz\hh]\hs\phi\hs^{\p\hn-\nh4}\hh d\phi$, where 
$\,\varLambda=h(D\phi,D\phi)$,
\item $(\varPi\nh,\h)\,$ is an Ein\-stein manifold with the Ein\-stein 
constant $\,\gz$,
\end{enumerate}
$\mathrm{Ric}^h$ and $\,\mathrm{div}^h$ being the Ric\-ci tensor of $\,h\,$ 
and the $\,h$-di\-ver\-gence, and $\,D\,$ denoting both the Le\-vi-Ci\-vi\-ta 
connection of $\,(M\nh,h)\,$ and the $\,h$-grad\-i\-ent.

Let us point out that, except for notations, (iii) and (i) are precisely 
(a)\hs--\hs(b) in \cite[Lemma 1.2]{derdzinski-piccione-20}, while (ii), with 
$\,\varLambda=h(D\phi,D\phi)$, is equivalent to the condition
\begin{equation}\label{ftd}
\phi^3\mathrm{div}^h(\phi^{-\nnh1}Dd\phi)
=[(\p\hh-\nh1)\varLambda-\gz\hh]\,d\phi+(1\nh-\p)\phi\,d\varLambda/2
\end{equation}
of \cite[Lemma 1.2(c)]{derdzinski-piccione-20}, as one sees differentiating by 
parts, and also to the equality
\begin{equation}\label{rdf}
\phi^2[\mathrm{Ric}^h(D\phi\nh,\,\cdot\,)
+d\hh\Delta\nnh^h\nnh\phi\hs]=[(\p\hh-\nh1)h(D\phi,D\phi)-\gz\hh]\,d\phi
+(1\nh-\p/2)\hs\phi\hskip1ptd\hh[h(D\phi,D\phi)]\hh,
\end{equation}
where $\,\Delta\nnh^h$ denotes the $\,h$-\nh La\-plac\-i\-an. See 
\cite[Lemma 1.2(e)]{derdzinski-piccione-20}.

Using, for instance, the components of the Ric\-ci tensor of 
$\,h+\nh\phi^2\h\,$ evaluated in \cite[the Appendix]{derdzinski-piccione-20},
and noting that, if $\,\varLambda=h(D\phi,D\phi)$,
\begin{equation}\label{fdf}
2\phi^{-\nnh1}\nnh\Delta\nnh^h\nnh\phi\,
+\,(\p\hh-\nh1)\phi^{-\nh2}\nh\varLambda\,
=\,4(p+\hn1)^{-\nnh1}\phi^{-(p+1)/2}\nh\Delta\nnh^h\nh\phi^{(p+1)/2}\hh,
\end{equation}
we express the (necessarily constant) scalar curvature $\,\lz\,$ of
$\,h+\nh\phi^2\h\,$ as follows:
\begin{equation}\label{csc}
\scal^h\hs+\,p\hh[\gz\phi^{-\nh2}\nh
-4(p+\hn1)^{-\nnh1}\phi^{-(\p+1)/2}\nh\Delta\nnh^h\nh\phi^{(\p+1)/2}]\,
=\,\lz\in\bbR\hh.
\end{equation}
Constancy of the scalar curvature $\,\lz\,$ is a general property of every 
metric with harmonic curvature \cite[Sect.~16.4(ii)]{besse}. Here we can also 
derive it from (i) and (\ref{rdf}): any $\,h$-Co\-daz\-zi tensor $\,b\,$ 
obviously has $\,\mathrm{div}^hb=d(\mathrm{tr}_h\w b)\,$ which, in the case of
$\,b=\mathrm{Ric}^h\nh-\hh\p\hh\phi^{-\nnh1}Dd\phi\,$ amounts to 
$\,-\nh2\p\hskip1.2pt\mathrm{div}^h(\phi^{-\nnh1}Dd\phi)
=d\hh[\scal^h\nh-2\p\phi^{-\nnh1}\nnh\Delta\nnh^h\nnh\phi]$, where we
have used the Bian\-chi identity for the Ric\-ci tensor
\cite[Proposition 1.94]{besse}. At the same time, (\ref{ftd}) states that 
$\,-\nh2\p\hskip1.2pt\mathrm{div}^h(\phi^{-\nnh1}Dd\phi)\,$ equals $\,\p\,$ 
times the differential of 
$\,[(\p\hh-\nh1)h(D\phi,D\phi)-\gz\hh]\hs\phi^{-\nh2}\nh$. Subtracting these
two equalities one gets
\begin{equation}\label{cte}
\mathrm{constancy\ of\ }\,\scal^h\nh
-2\p\phi^{-\nnh1}\nnh\Delta\nnh^h\nnh\phi-\p\hh[(\p\hh-\nh1)h(D\phi,D\phi)
-\gz\hh]\hs\phi^{-\nh2},
\end{equation}
that is, by 
(\ref{fdf}) -- (\ref{csc}), of $\,\lz$. Next, for a warped product (\ref{nwp})
with $\,\mathrm{div}\,R=0$,
\begin{equation}\label{ana}
\mathrm{the\ base\ metric\ }\,h\,\mathrm{\ is\ real}\hyp\mathrm{analytic\ in\
suitable\ local\ coordinates.}
\end{equation}
In other words, the $\,C^\infty\nnh$-man\-i\-fold structure of the base
$\,M\,$ contains a unique real-an\-a\-lyt\-ic structure making $\,h\,$
analytic. Namely, as shown by DeTurck and Goldschmidt
\cite{deturck-goldschmidt}, the analog of (\ref{ana}) holds for
har\-mon\-ic-cur\-va\-ture metrics, while the base $\,(M\nh,h)$ is isometric
to a totally geodesic sub\-man\-i\-fold of the warped product (\ref{nwp}).
\begin{remark}\label{surfc}If $\,\dim M=2$, condition (i) amounts to 
\hbox{$\,(dK\hn^h\nh+\hh\p K\hn^h\nh\phi^{-\nnh1}\nh d\phi)\wedge h$} 
$+\p\hh\phi^{-\nh2}\nh d\phi\wedge Dd\phi=0$, as one sees using (\ref{cod})
and (\ref{dfa}.a) for the pair $\,(f\nh,\bz)\,$ equal to $\,(K\hn^h\nh,h)\,$
or $\,(\phi^{-\nnh1}\nh,Dd\phi)$, with $\,\mathrm{Ric}^h\nh=K\hn^h\nh h$,
followed by (\ref{dfa}.b).
\end{remark}
\begin{remark}\label{parts}Any nontrivial warped product (\ref{nwp}) with a 
compact base $\,(M\nh,h)\,$ and harmonic curvature has $\,\p=\dim\varPi\ge2$,
as the Ein\-stein constant $\,\gz\,$ of the fibre $\,(\varPi\nh,\h)\,$ must be
positive \cite[Theorem 1.4]{derdzinski-piccione-20}: the $\,h$-in\-ner product
of the left-hand side of (ii) with the $\,h$-grad\-i\-ent $\,D\phi\,$
obviously differs by an $\,h$-di\-ver\-gence from $\,-\phi\hs^{\p\hn-\nh2}$ times the $\,h$-norm squared of $\,Dd\phi\,$ while, with 
$\,\varLambda=h(D\phi,D\phi)$, the analogous inner product for the right-hand side 
equals $\,[(\p\hh-\nh1)\varLambda-\gz\hh]\varLambda\nh\phi\hs^{\p\hn-\nh4}\nh$.
\end{remark}
\begin{remark}\label{pgron}
The case of a one\hs-di\-men\-sion\-al fibre (\hs$\p=1$), for nontrivial 
warped products with harmonic curvature, is of very limited interest: it 
precludes compactness of the base (Remark~\ref{parts}) and, for 
two\hs-di\-men\-sion\-al bases -- the main focus of this paper -- the 
resulting three-man\-i\-folds (\ref{nwp}) are con\-for\-mal\-ly flat
\cite[Sect.~16.4(e)]{besse}.
\end{remark}

\section{Warped products with two\hs-di\-men\-sion\-al bases\done}\label{wt}
\setcounter{equation}{0}
Recall that the \wp\ metric in (\ref{nwp}) is con\-for\-mal to a product 
metric:
\begin{equation}\label{hpf}
h+\nh\phi^2\h\hs\,\,=\,\,\hs[\hs g+\h\hs]/\vt^2,\hskip8pt\mathrm{where\ }
\,\,g\,=\,\phi^{-\nh2}\nh h\,\,\mathrm{\ and\ }\,\vt=1/\phi\hh.
\end{equation}
The question of which nontrivial warped products (\ref{nwp}) with 
two\hs-di\-men\-sion\-al bases have harmonic curvature may obviously be
rephrased in terms of the surface metric $\,g=\phi^{-\nh2}\nh h\,$ and the 
function $\,\vt=1/\phi$. Remark~\ref{pgron} and condition (iii) of
Section~\ref{wp} allow us to assume that the fibre $\,(\varPi\nh,\h)\,$ is 
an Ein\-stein manifold of dimension $\,\p\ge2$ with some Ein\-stein constant 
$\,\gz$. In Section~\ref{pt} we prove the following result.
\begin{theorem}\label{srfbs}
Given a Riemannian surface\/ $\,(M\nh,g)$, a nonconstant function 
$\,\vt:M\to(0,\infty)$, and an Ein\-stein manifold\/ $\,(\varPi\nh,\h)\,$ of 
dimension\/ $\,\p\ge2\,$ with the Ein\-stein constant $\,\gz$, the 
warp\-ed-\hn prod\-uct metric\/ $\,[\hs g+\h\hs]/\vt^2$ on\/ 
$\,M\times\hh\varPi\,$ has harmonic curvature if and only if the Gauss\-i\-an 
curvature\/ $\,K\nh$ of\/ $\,g\,$ satisfies the equation
\begin{equation}\label{rwr}
(2K\nnh+\p\hs\gz)\vt^2\nh+2(\p+\hn1)\vt\hs\Delta\vt
-(\p+\hn1)(\p+\hn2)\hh g(\nabla\hskip-1.1pt\vt,\nabla\hskip-1.1pt\vt)\,=\,\lz
\end{equation}
for a constant\/ $\,\lz\in\bbR$, and one of the following two conditions 
occurs.
\begin{enumerate}
  \def\theenumi{{\rm\alph{enumi}}}
\item $K\,$ is constant, and equal to\/ $\,-\gz/\nh(\p\hh-\nh1)$,
\item $K\,$ is nonconstant, $\,(\p\hh-\nh1)K\nnh+\gz\ne0\,$ everywhere in\/ 
$\,M\nh$, and\/ $\,\vt\,$ equals a positive constant times 
$\,|\hh(\p\hh-\nh1)K\nnh+\gz\hs|^{1/(\p\hh-\nh1)}\nh$.
\end{enumerate}
The constant\/ $\,\lz\,$ in\/ {\rm(\ref{rwr})} then coincides with the scalar 
curvature of\/ $\,[\hs g+\h\hs]/\vt^2\nh$.
\end{theorem}
The positive constant mentioned of Theorem~\ref{srfbs}(b) may always be 
assumed equal to $\,1\,$ by simultaneously rescaling $\,\vt\,$ and $\,\lz$, 
so that (\ref{rwr}) still holds. The resulting {\it normalized version\/} of 
case (b) in Theorem~\ref{srfbs} amounts to a condition imposed on $\,K\hs$ 
alone, with no reference to $\,\vt\,$ at all. Explicitly, it reads
\begin{equation}\label{csb}
(\p+\hn1)\hs[2\hs\omega\hh\Delta\hn K\,
-\,(3\p\hh-\nh2)\hs g(\nabla\nnh K\nh,\nnh\nabla\nnh K)]\,
=\,\lz\hs|\hh\omega|^{2(\p\hh-\nh2)/(\p\hh-\nh1)}\hs
-\,(2K\nnh+\p\hs\gz)\hh\omega^2
\end{equation}
for $\,\omega=(\p\hh-\nh1)K\nnh+\gz$, with constants $\,\gz,\lz\in\bbR$, where 
$\,K\,$ is the (nonconstant) Gauss\-i\-an curvature of the Riemannian surface 
$\,(M\nh,g)$, and $\,\omega\ne0\,$ everywhere.
\begin{theorem}\label{uniqu}
Under the assumptions stated in the preceding three lines, 
if\/ $\hs M$ \hbox{is compact,} $\,\gz\,$ and\/ $\,\lz\,$ are uniquely
determined by\/ $\hs\,g\,$ and\/ $\,\p$.
\end{theorem}
Theorem~\ref{uniqu}, which will be proved in Section~\ref{pf}, has an obvious
consequence: the product $\,\gz\mathrm{A}$, for
$\,\mathrm{A}=\mathrm{area}\hs(M\nh,g)$, is a homothety invariant of $\,g$.
Note that multiplying $\,g\,$ by $\,z\in(0,\infty)\,$ causes the quintuple 
$\,(\mathrm{A},\nh K\nh,\gz,\omega,\lz)\,$ to be replaced with 
$\,(z\hn\mathrm{A},z^{-\nnh1}\nnh K\nh,z^{-\nnh1}\nh\gz,z^{-\nnh1}\hn\omega,
z^{(1+\p)/(1-\p)}\nh\lz)$.      
\begin{remark}\label{hminv}Another homothety invariant, naturally associated 
with any nonflat compact Riemannian surface having the Gauss\-i\-an curvature 
$\,K\nh$, is the point
$\,[K\nnh_{\mathrm{max}}\w\hskip-3pt:\hskip-3ptK\nnh_{\mathrm{min}}\w]\,$
of the real projective line $\,\bbR\mathrm{P}^1\nh$, where
$\,[\hskip3pt:\hskip3pt]\,$ are the homogeneous coordinates. Clearly, $\,K\hs$
is constant if and only if
$\,[K\nnh_{\mathrm{max}}\w\hskip0pt:\hskip0ptK\nnh_{\mathrm{min}}\w]
=[1\hskip-1.9pt:\hskip-2pt1]$.
\end{remark}
\begin{remark}\label{empos}Whenever the hypotheses of Theorem~\ref{srfbs} are 
satisfied, along with (\ref{rwr}) for a constant $\,\lz$, and one of 
conditions (a) -- (b) holds, compactness of $\,M\,$ implies positivity of both 
$\,\gz\,$ and $\,\lz$. See Remarks~\ref{parts} and~\ref{postv}.
\end{remark}
\begin{remark}\label{cnffl}In the context of Question~\ref{which}, case (a) of 
Theorem~\ref{srfbs} is of interest only for $\,\p\ge4$, since an Ein\-stein 
manifold $\,(\varPi\nh,\h)\,$ of dimension $\,\p\in\{2,3\}$ with the 
Ein\-stein constant $\,\gz\,$ has constant sectional curvature 
$\,\gz/\nh(\p\hh-\nh1)$. According to Remark~\ref{cnfpr}, this implies
con\-for\-mal flatness of the har\-mon\-ic-cur\-va\-ture metric
$\,[\hs g+\h\hs]/\vt^2$ (while, in case (b), $\,[\hs g+\h\hs]/\vt^2$ is never
con\-for\-mal\-ly flat).
\end{remark}

\section{Proof of Theorem~\ref{srfbs}\done}\label{pt}
\setcounter{equation}{0}
Due to (\ref{hpf}), 
$\,[\hs g+\h\hs]/\vt^2$ has harmonic curvature if and only if $\,(M\nh,h)\,$ 
satisfies (i) and (ii) in Section~\ref{wp} or, equivalently, (i) and 
(\ref{rdf}). This further amounts to
\begin{equation}\label{bec}
\begin{array}{rl}
\mathrm{a)}&
\phi^2dK\hn^h\nnh\,
+\,p\hs[K\hn^h\hn\phi\,d\phi+(\Delta\nnh^h\nnh\phi)\hs d\phi-d\varLambda\nnh/2\hh]\,
=\,0\hh,\\
\mathrm{b)}&
\phi^2(K\hn^hd\phi+d\hh\Delta\nnh^h\nnh\phi)\,
-\,[(\p\hh-\nh1)\varLambda-\gz\hh]\,d\phi\,
-\,(1\nh-\p/2)\phi\,d\varLambda\,=\,0\hh,
\end{array}
\end{equation}
with $\,\varLambda=h(D\phi,D\phi)$, and $\,K\hn^h$ denoting the Gauss\-i\-an curvature 
of $\,h$. Namely, $\,\mathrm{Ric}^h\nh=K\hn^h\nh h$, so that (\ref{rdf}) 
becomes (\ref{bec}.b), while Remark~\ref{surfc} and (\ref{jxh}) easily yield 
the equivalence between (i) and (\ref{bec}.a). 

As a consequence of (\ref{bec}), we obtain the relation (\ref{cte}), which now 
reads
\begin{equation}\label{khp}
2K\hn^h\nh+p\hh[\gz\phi^{-\nh2}\nh-2\phi^{-\nnh1}\nnh\Delta\nnh^h\nnh\phi
-(\p\hh-\nh1)\phi^{-\nh2}\nh\varLambda]=\lz\,\mathrm{\ for\ some\ 
}\,\lz\in\bbR\hh.
\end{equation}
Explicitly, subtracting (\ref{bec}.b) multiplied by $\,2\p\phi^{-\nh3}$ from
$\,2\phi^{-\nh2}$ times (\ref{bec}.a), we see that 
$\,d\hh\{\ldots\}=0$, with $\,\{\ldots\}\,$ denoting the left-hand side in 
(\ref{khp}). The system (\ref{bec}) is thus equivalent to one consisting of
(\ref{bec}.a) and (\ref{khp}), namely
\begin{equation}\label{eqv}
\begin{array}{rl}
\mathrm{i)}&
\phi^2dK\hn^h\nnh\,
+\,p\hs[K\hn^h\hn\phi\,d\phi+(\Delta\nnh^h\nnh\phi)\hs d\phi
-d\varLambda/2\hh]\,
=\,0\hh,\\
\mathrm{ii)}&
2K\hn^h\nh+p\hh[\gz\phi^{-\nh2}\hs-\,2\phi^{-\nnh1}\nnh\Delta\nnh^h\nnh\phi\,
-\,(\p\hh-\nh1)\phi^{-\nh2}\nh\varLambda]\,\mathrm{\ is\ constant.}
\end{array}
\end{equation}
Let us now rewrite (\ref{eqv}) in terms of the con\-for\-mal\-ly related 
metric $\,g=\phi^{-\nh2}\nh h\,$ and the function $\,\vt=1/\phi\,$ on 
$\,M\nh$, using the symbols $K,\nabla\nh,\Delta\,$ for the Gauss\-i\-an 
curvature of $\,g$, the $\,g$-grad\-i\-ent, and the 
$\,g\hh$-\nh La\-plac\-i\-an, as well as setting 
$\,Q=g(\nabla\hskip-1.1pt\vt,\nabla\hskip-1.1pt\vt)\,$ and 
$\,Y\nh=\Delta\vt$. Since $\,\varLambda\nh=Q/\hn\vt^2$ and 
$\,\Delta(1/\hn\vt)=(2\hh Q-\vt Y)/\hn\vt^3\nh$, Remark~\ref{sccrv} yields
\begin{equation}\label{tqd}
\begin{array}{rl}
\mathrm{a)}&
\vt^3dK\nnh+\vt^2dY\nh-(1+\p/2)\vt\hs d\hh Q
+[(2-\p)K\vt+Y]\hh\vt\,d\vt\,=\,0\hh,\\
\mathrm{b)}&
(2K\nnh+p\gz)\vt^2\nh+2(\p+\hn1)\vt Y\nh
-(\p+\hn1)(\p+\hn2)\hh Q\,\mathrm{\ is\ constant.}
\end{array}
\end{equation}
Finally, we may replace (\ref{tqd}) with the (obviously equivalent) system 
consisting of (\ref{tqd}.b) and the equality 
$\,2\p(\p\hh-\nh1)^{-\nnh1}\hn\vt^{\p+\hn2}
d\hs\{\hn\vt^{1\nh-\p}[(\p\hh-\nh1)K\nnh+\gz\hh]\}=0\,$ obtained by applying 
$\,d\,$ to (\ref{tqd}.b), multiplying the result by $\,\vt$, and then 
subtracting it from $\,2(\p+\hn1)\,$ times (\ref{tqd}.a). This proves 
Theorem~\ref{srfbs}, with the cases (a), (b) depending on whether the constant 
$\,\vt^{1\nh-\p}[(\p\hh-\nh1)K\nnh+\gz\hh]\,$ is or is not equal to $\,0$. 

\section{Proof of Theorem~\ref{uniqu}\done}\label{pf}
\setcounter{equation}{0}
We assume that $\,\p>2$, as the case $\,\p=2\,$ is already settled in 
\cite[Remark 4]{derdzinski-88}.

It suffices to establish uniqueness of  $\,\gz$, since (\ref{csb}) provides an 
expression for $\,\lz\,$ in terms of $\,\gz\,$ and geometric invariants of 
$\,g$. Suppose that, on the contrary, in addition to (\ref{csb}) with 
$\,\omega=(\p\hh-\nh1)K\nnh+\gz\,$ one also has
\[
(\p+\hn1)\hs[2\hs\tilde\omega\hh\Delta\hn K\,
-\,(3\p\hh-\nh2)\hs g(\nabla\nnh K\nh,\nnh\nabla\nnh K)]\,
=\,\tilde\lz\hs|\hs\tilde\omega|^{2(\p\hh-\nh2)/(\p\hh-\nh1)}\hs
-\,(2K\nnh+\p\hs\tilde\gz)\hh\tilde\omega^2
\]
for $\,\tilde\omega=(\p\hh-\nh1)K\nnh+\tilde\gz\,$ and constants 
$\,\tilde\gz,\tilde\lz$, while $\,\tilde\gz<\gz\,$ and 
$\,\omega\hh\tilde\omega\ne0\,$ everywhere, cf.\ the lines following
(\ref{csb}). As $\,\omega-\tilde\omega=\gz-\tilde\gz$, subtracting the last
equality (or, $\,\omega\,$ times it) from (\ref{csb}) or, respectively, from
$\,\tilde\omega\,$ times (\ref{csb}), one gets $\,\Delta\hn K=\varSigma(K)$
and $\,g(\nabla\nnh K,\nnh\nabla\nnh K)=2Z(K)\,$ with the functions
$\,\varSigma\,$ and $\,Z$ of the variable $\,K\hs$ given by
\begin{equation}\label{sgz}
\begin{array}{rl}
\mathrm{i)}&
2(\p+\hn1)(\gz-\tilde\gz)\hh\varSigma(K)=
\tilde\omega^2\nh\tilde\varTheta-\omega^2\nh\varTheta\hs,\\
\mathrm{ii)}&
2(\p+\hn1)(3\p\hh-\nh2)(\gz-\tilde\gz)Z(K)
=\omega\hh\tilde\omega\hs[\hs\tilde\omega\tilde\varTheta
  -\omega\varTheta]\hh,\hskip6pt\mathrm{for}\\
\mathrm{iii)}&\omega=(\p\hh-\nh1)K\nnh+\gz\hh,\hskip16pt
\tilde\omega=(\p\hh-\nh1)K\nnh+\tilde\gz\hh,\hskip6pt\mathrm{and}\\
\mathrm{iv)}&\varTheta=2K\nnh+\p\hs\gz-\lz\hs|\hh\omega|^{2/(1\nh-\p)}\hh,
\hskip16pt\tilde\varTheta=2K\nnh+\p\hs\tilde\gz
-\tilde\lz\hs|\hs\tilde\omega|^{2/(1\nh-\p)}\nh.
\end{array}
\end{equation}
Thus, since $\,\omega\hh\tilde\omega\ne0\,$ everywhere,
\begin{equation}\label{rng}
\begin{array}{l}
\mathrm{the\ values\ of\ }\hs K\mathrm{\ lie\ in\
}\hs I\nnh_*\w\hn\mathrm{,\ which\ is\ one\ of\ the\ intervals}\\
(-\infty,\gz/(1-\p))\hh,\,(\gz/(1-\p),\tilde\gz/(1-\p))\hh,
\,(\tilde\gz/(1-\p),\infty)\nh,\\
\mathrm{while\ }\,\,\hs\omega,\tilde\omega,\varTheta\nh,\tilde\varTheta\nh,
\varSigma\,\mathrm{\ and\ }\hs\,Z\,\mathrm{\ are\
real}\hyp\mathrm{an\-a\-lyt\-ic\ functions}\\
\mathrm{of\ the\ real\ variable\ \nh}K\nnh,\nnh\mathrm{\ defined\ on\ the\
whole\ interval\ \nh}I\nnh_*\w\nh.
\end{array}
\end{equation}
As $\,(\p\hh-\nh1)K\nnh=\omega-\gz=\tilde\omega-\tilde\gz$, (\ref{sgz}.iii) --
(\ref{sgz}.iv), with $\,(\hskip2.3pt)'\nh=d/dK\nh$, give 
\begin{equation}\label{omp}
\begin{array}{l}
(\p\hh-\nh1)\varTheta=2\hs\omega+(\p+\nh1)(\p\hh-\nh2)\hh\gz
-(\p\hh-\nh1)\lz\hs|\hh\omega|^{2/(1\nh-\p)},\\
(\p\hh-\nh1)\tilde\varTheta=2\hs\tilde\omega+(\p+\nh1)(\p\hh-\nh2)\hh\tilde\gz
-(\p\hh-\nh1)\tilde\lz\hs|\hs\tilde\omega|^{2/(1\nh-\p)}\nh,\\
\omega'\nh=\tilde\omega'\nh=\p\hh-\nh1\hh,\hskip8.2pt
\varTheta'\nh=2
+2(\mathrm{sgn}\,\omega)\lz\hh|\hh\omega|^{(1+\p)/(1\nh-\p)},\\
\tilde\varTheta'\nh=2
+2(\mathrm{sgn}\,\tilde\omega)\tilde\lz\hh|\hs\tilde\omega\hh|^{(1+\p)/(1\nh-\p)}.
\end{array}
\end{equation}
Multiplying (\ref{sgz}.i) and (\ref{sgz}.ii) by $\,\p\hh-\nh1$, then using 
(\ref{omp}), we obtain
\begin{equation}\label{psm}
\begin{array}{rl}
\mathrm{i)}&
2(\p^2\nnh-\nh1)(\gz-\tilde\gz)\hh\varSigma
=2(\tilde\omega^3\nh-\omega^3)
+(\p+\nh1)(\p\hh-\nh2)(\tilde\gz\hs\tilde\omega^2\nh-\gz\hs\omega^2)\\
&\hskip52pt+\,(\p\hh-\nh1)[\lz\hs|\hh\omega|^{2(\p\hh-\nh2)/(\p\hh-\nh1)}\nh
-\tilde\lz\hs|\hs\tilde\omega|^{2(\p\hh-\nh2)/(\p\hh-\nh1)}]\nh,\\
\mathrm{ii)}&2(\p^2\nnh-\nh1)(3\p\hh-\nh2)(\gz-\tilde\gz)Z\\
&\hskip41.8pt=\,[2\hs\tilde\omega^3\nh
+(\p+\nh1)(\p\hh-\nh2)\hh\tilde\gz\hs\tilde\omega^2\nh
-(\p\hh-\nh1)\tilde\lz\hs|\hs\tilde\omega|^{2(\p\hh-\nh2)/(\p\hh-\nh1)}]\hs\omega\\
&\hskip41.8pt-\,\,[2\hs\omega^3\nh+(\p+\nh1)(\p\hh-\nh2)\hh\gz\hh\omega^2\nh
-(\p\hh-\nh1)\lz\hs|\hh\omega|^{2(\p\hh-\nh2)/(\p\hh-\nh1)}]\hs\tilde\omega\nh.
\end{array}
\end{equation}
Next, as a consequence of (\ref{omp}),
\begin{equation}\label{oth}
\begin{array}{l}
[\hh\omega\varTheta]'\nh=4\hs\omega+(\p+\nh1)(\p\hh-\nh2)\hh\gz
-(\p-\hn3)\lz\hs|\hh\omega|^{2/(1\nh-\p)},\\
{}[\hh\tilde\omega\tilde\varTheta]'\nh=4\hs\tilde\omega
+(\p+\nh1)(\p\hh-\nh2)\hh\tilde\gz
-(\p-\hn3)\tilde\lz\hs|\hs\tilde\omega|^{2/(1\nh-\p)}.
\end{array}
\end{equation}
Thus, $\,[\hh\omega^2\nh\varTheta]'\nh
=\omega'\omega\varTheta+\omega[\hh\omega\varTheta]'\nh
=\omega\hh\{(\p\hh-\nh1)\varTheta+[\hh\omega\varTheta]'\}$. Now, by 
(\ref{omp}) and (\ref{oth}),
\begin{equation}\label{ost}
\begin{array}{l}
[\hh\omega^2\varTheta]'\nh=2[3\hs\omega+(\p+\nh1)(\p\hh-\nh2)\hh\gz
-(\p\hh-\nh2)\lz\hs|\hh\omega|^{2/(1\nh-\p)}]\hs\omega,\\
{}[\hh\tilde\omega^2\tilde\varTheta]'\nh=2[3\hs\tilde\omega
+(\p+\nh1)(\p\hh-\nh2)\hh\tilde\gz
-(\p\hh-\nh2)\tilde\lz\hs|\hs\tilde\omega|^{2/(1\nh-\p)}]\hs\tilde\omega.
\end{array}
\end{equation}
From (\ref{sgz}.ii), 
$\,2(\p+\hn1)(3\p\hh-\nh2)(\gz-\tilde\gz)Z'\nh
=\{\omega\hh[\hh\tilde\omega^2\tilde\varTheta]-
\tilde\omega\hh[\hh\omega^2\varTheta]\}'\nh$. The Leib\-niz rule applied to 
both products $\,\omega\hh[\hh\tilde\omega^2\tilde\varTheta]\,$ and 
$\,\tilde\omega\hh[\hh\omega^2\varTheta]\,$ yields, via (\ref{omp}), 
(\ref{sgz}.i) and (\ref{ost}),
\[
\begin{array}{l}
2(\p+\hn1)(3\p\hh-\nh2)(\gz-\tilde\gz)Z'\nh
=2(\p^2\nnh-\nh1)(\gz-\tilde\gz)\hh\varSigma\\
\hskip53.8pt+\,2\hs\omega\hh\tilde\omega\hs[(\p^2\nnh-\p+\hn1)(\tilde\gz-\gz)
+(\p\hh-\nh2)\lz\hs|\hh\omega|^{2/(1\nh-\p)}\nh
-(\p\hh-\nh2)\tilde\lz\hs|\hs\tilde\omega|^{2/(1\nh-\p)}]\hh,
\end{array}
\]
as $\,\omega-\tilde\omega=\gz-\tilde\gz$, which, setting
$\,\gamma=\mathrm{sgn}\,\omega\,$ and 
$\,\tilde\gamma=\mathrm{sgn}\,\tilde\omega$, we can rewrite as
\begin{equation}\label{zpr}
\begin{array}{l}
2(\p+\hn1)(\gz-\tilde\gz)[(3\p\hh-\nh2)Z'\nh-(\p\hh-\nh1)\hh\varSigma]
=2\hh(\p^2\nnh-\p+\hn1)(\tilde\gz-\gz)\hs\omega\hh\tilde\omega\\
\hskip58pt+\,2(\p\hh-\nh2)[\gamma\lz\hs\tilde\omega\hs|\hh\omega|^{(\p-\hn3)/(\p\hh-\nh1)}\nh
-\tilde\gamma\tilde\lz\hs\omega\hs|\hs\tilde\omega|^{(\p-\hn3)/(\p\hh-\nh1)}]
\hh.
\end{array}
\end{equation}
With the same meaning of $\,\gamma\,$ and $\,\tilde\gamma$, using
(\ref{sgz}.i) and (\ref{ost}), or (\ref{omp}), we get
\begin{equation}\label{spe}
\begin{array}{rl}
\mathrm{i)}&(\p+\hn1)(\gz-\tilde\gz)\hh\varSigma'\nh
=3(\tilde\omega^2\nh-\omega^2)
+(\p+\nh1)(\p\hh-\nh2)(\tilde\gz\hs\tilde\omega-\gz\hs\omega)\\
&\phantom{(\p+\hn1)(\gz-\tilde\gz)\hh\varSigma'\nh}+\,(\p\hh-\nh2)[\gamma\lz\hh|\hh\omega|^{(\p-\hn3)/(\p\hh-\nh1)}\nh
-\tilde\gamma\tilde\lz\hh|\hs\tilde\omega|^{(\p-\hn3)/(\p\hh-\nh1)}]\hh,\\
\mathrm{ii)}&\varTheta''\nh=-\nh2(\p+\hn1)\lz\hh|\hh\omega|^{2\p/(1\nh-\p)},\hskip8.2pt
\tilde\varTheta''\nh
=-\nh2(\p+\hn1)\tilde\lz\hh|\hs\tilde\omega\hh|^{2\p/(1\nh-\p)}.
\end{array}
\end{equation}
Consequently, (\ref{sgz}.ii) gives 
$\,2(\p+\hn1)(3\p\hh-\nh2)(\gz-\tilde\gz)Z''\nh
=[(\omega\hs\tilde\omega^2)\tilde\varTheta]''\nh
-[(\omega^2\tilde\omega)\varTheta]''\nh
=(\omega\hs\tilde\omega^2)\tilde\varTheta''\nh
-(\omega^2\tilde\omega)\varTheta''\nh
+2(\omega\hs\tilde\omega^2)'\tilde\varTheta'\nh
-2(\omega^2\tilde\omega)'\varTheta'\nh
+(\omega\hs\tilde\omega^2)''\tilde\varTheta
-(\omega^2\tilde\omega)''\varTheta$. As (\ref{omp}) easily implies that 
$\,(\omega\hs\tilde\omega^2)'\nh,(\omega^2\tilde\omega)'\nh,
(\omega\hs\tilde\omega^2)''$ and $\,(\omega^2\tilde\omega)''$ are, 
respectively, equal to 
$\,(\p\hh-\nh1)(\tilde\omega+2\hs\omega)\hs\tilde\omega$, 
$\,(\p\hh-\nh1)(\omega+2\hs\tilde\omega)\hs\omega$, 
$\,2(\p\hh-\nh1)^2(2\hs\tilde\omega+\omega)\,$ and 
$\,2(\p\hh-\nh1)^2(2\hs\omega+\tilde\omega)$, we obtain 
$\,2(\p+\hn1)(3\p\hh-\nh2)(\gz-\tilde\gz)Z''\nh
=(\omega\hs\tilde\omega^2)\tilde\varTheta''\nh
-(\omega^2\tilde\omega)\varTheta''\nh
+2(\p\hh-\nh1)(\tilde\omega+2\hs\omega)\hs\tilde\omega\hs\tilde\varTheta'\nh
-2(\p\hh-\nh1)(\omega+2\hs\tilde\omega)\hs\omega\hs\varTheta'\nh
+2(\p\hh-\nh1)(2\hs\tilde\omega+\omega)[(\p\hh-\nh1)\tilde\varTheta]
-2(\p\hh-\nh1)(2\hs\omega+\tilde\omega)[(\p\hh-\nh1)\varTheta]$. Therefore, 
replacing $\,\tilde\varTheta''\nh,\varTheta''\nh,\tilde\varTheta'\nh,
\varTheta'\nh,(\p\hh-\nh1)\tilde\varTheta\,$ and $\,(\p\hh-\nh1)\varTheta\,$ 
with the expressions provided by (\ref{spe}.ii) and (\ref{omp}), we 
see that, since $\,\omega-\tilde\omega=\gz-\tilde\gz$,
\begin{equation}\label{zpp}
\begin{array}{l}
(\p+\hn1)(3\p\hh-\nh2)Z''\,
+\,\,3(\p\hh-\nh1)(\p^2\nnh-\p+\hn2)\hs\omega\,\,\mathrm{\hs\ equals\hs\ the\hs\ sum\hs\ of\hs\ a}\\
\mathrm{constant\hs\ and\hs\ a\hs\hs\ constant
}\hn\hyp\hh\mathrm{coefficient\hs\ combination\ of\hs\ the\hs\ functions}\\
\lz\hs\tilde\omega\hs|\hh\omega|^{2/(1\nh-\p)}\nh
-\tilde\lz\hs\omega\hs|\hs\tilde\omega|^{2/(1\nh-\p)}\nh\mathrm{\ and\
}\gamma\lz\hh|\hh\omega|^{(\p-\hn3)/(\p\hh-\nh1)}\nh
-\tilde\gamma\tilde\lz\hh|\hs\tilde\omega|^{(\p-\hn3)/(\p\hh-\nh1)}.
\end{array}
\end{equation}
\begin{lemma}\label{limit}As\/ $\,K\nh\to\pm\infty$, one has the following 
limit relations.
\begin{enumerate}
  \def\theenumi{{\rm\alph{enumi}}}
\item $2(\p+\hn1)\hh\varSigma/K^2$ and\/ 
$\,(\p+\hn1)\hh\varSigma'\nnh/K\hs$ both tend to\/ 
$\,-(\p\hh-\nh1)(\p^2\nnh-\p+\hn4)$,
\item $2(\p+\hn1)(3\p\hh-\nh2)Z/\nh K^3\to-(\p\hh-\nh1)^2(\p^2\nnh-\p+\hn2)$,
\item $2(\p+\hn1)(3\p\hh-\nh2)Z'\nnh/K^2
\to-3(\p\hh-\nh1)^2(\p^2\nnh-\p+\hn2)$,
\item $(\p+\hn1)(3\p\hh-\nh2)Z''\nnh/K\to-3(\p\hh-\nh1)(\p^2\nnh-\p+\hn2)$.
\end{enumerate}
The limits, as\/ $K\nh\to\pm\infty$, \,of\/ 
$4[(\p+\hn1)(3\p\hh-\nh2)]^2(2Z'\nh-\varSigma)(Z'\nh-\varSigma)/K^4$ and of\/ 
$\,8[(\p+\hn1)(3\p\hh-\nh2)]^2(Z''\nh-\hs\varSigma'-\hs K)\hh Z/K^4\hs$ are\/  
$\,-(\p\hh-\nh1)^2(\p\hh-\nh2)(\p^2\nnh+5\p-\nh2)(3\p^2\nnh-\p+\nh2)$ and,
respectively, $\,-4(\p\hh-\nh1)^2(\p\hh-\nh2)
(\p^2\nnh-\p+\hn2)(3\p^3\nnh-\nh5\p^2\nnh+12\p\nh-\nh8)$. The difference
between the former and the latter limits equals\/
$\,(\p\hh-\nh1)^3(\p\hh-\nh2)\,$ times the positive function\/ 
$\,12\p^4\nnh-\nh23\p^3\nnh+55\p^2\nnh-\nh56\p+60\,$ of the real variable\/
$\,\p\ge\nh1$.
\end{lemma}
\begin{proof}By (\ref{sgz}.iii), $\,\omega/\nh K\hs$ and
$\,\tilde\omega/\nh K\hs$ tend to $\,\p\hh-\nh1\,$ as $\,K\nh\to\pm\infty$.
Since $\,\tilde\omega^3\nh-\omega^3\nh=(\tilde\omega-\omega)(\tilde\omega^2\nh
+\tilde\omega\hh\omega+\omega^2)\,$ and $\,\omega-\tilde\omega=\gz-\tilde\gz$, 
(\ref{psm}.i) and (\ref{spe}.i), with $\,2(\p\hh-\nh2)/(\p\hh-\nh1)<2$ and
$\,(\p-\hn3)/(\p\hh-\nh1)<1$, yield (a). Similarly, (\ref{psm}.ii),
(\ref{zpr}) combined with (a), and (\ref{zpp}) give (b), (c) and,
respectively, (d). Now (a) -- (d) imply positivity in the final clause as
$\hs12\p^4\nnh-\nh23\p^3\nnh+55\p^2\nnh-\nh56\p+60=
\p^2\nnh(\p\hh-\nh1)(12\p\hh-\nh11)+4(11\p^2\nnh-14\p+\hn15)$.
\end{proof}
We now derive the contradiction that proves Theorem~\ref{uniqu}. 
By Lemma~\ref{nabdk}(ii), our $\,Z\,$ and $\,\varSigma\,$ satisfy the
differential equation (\ref{zpm}), and so the interval $\,I\nnh_*\w$ in
(\ref{rng}) must be bounded since, due to the positivity claim at the very end
of Lemma~\ref{limit}, the two sides of (\ref{zpm}), divided by $\,K^4\nh$, have
different limits as $\,|K\hn|\to\infty$. (At the beginning of this section we
assumed that $\,\p>2$.) Thus, $\,I\nnh_*\w=(\gz/(1-\p),\tilde\gz/(1-\p))\,$
and $\,K\nh<\tilde\gz/(1-\p)$, so that $\,K\nnh<0$, as Remark~\ref{parts}
yields $\,\tilde\gz>0$. According to the lines preceding (\ref{ext}), our
equalities $\,\Delta\hn K=\varSigma(K)$ and
$\,g(\nabla\nnh K,\nnh\nabla\nnh K)=2Z(K)\,$ imply the existence of a
$\,g$-Kil\-ling field $\,v\,$ without zeros, defined on a nonempty connected
open set $\,\,U\subseteq M\nh$, which is also 
an $\,h$-Kil\-ling field, for the metric $\,h=g/\vt^2$ in (\ref{hpf}), since
the normalization of (\ref{csb}) gives 
$\,\vt=|\hh(\p\hh-\nh1)K\nnh+\gz\hs|^{1/(\p\hh-\nh1)}\nh$, that is,
$\,\vt=|\hh\omega|^{1/(\p\hh-\nh1)}\nh$, and the local flow
of $\,v\,$ preserves the Gauss\-i\-an curvature $\,K\nh$. The same obviously
applies to the metric $\,\tilde h=g/\tv^2\nh$, where 
$\,\tv=|\hh\tilde\omega|^{1/(\p\hh-\nh1)}\nh$. By (\ref{ana}) and
Theorem~\ref{srfbs}, $\,h\,$ and $\,\tilde h\,$ are real-an\-a\-lyt\-ic. At
least one of them has nonconstant Gaussian curvature. Otherwise, their
constant Gaussian curvatures would be negative (from the
\hbox{Gauss\hh-}\nh Bonnet theorem -- note that $\,K\nnh<0$) and, as $\,h\,$
and $\,\tilde h\,$ are con\-for\-mal\-ly related, Remark~\ref{ymuni} would
imply constancy of their con\-for\-mal factor $\,\tv/\hn\vt$, thus making
$\,K\hs$ constant. Finally, $\,\eu(M)<0\,$ since $\,K\nnh<0$, so that by
Lemma~\ref{glbki} a nontrivial Kil\-ling field exists on $\,(M\nh,h)$, or
$\,(M\nh,\tilde h)$, or on a two\hh-\hn fold isometric covering. This in turn
contradicts Remark~\ref{holom}.

\section{Theorem~\ref{srfbs}, rephrased\done}\label{tr}
\setcounter{equation}{0}
Let us rewrite Theorem~\ref{srfbs} in terms of the positive 
function $\,f\nh=\vt^{-\p/2}\nh$, the parameter $\,\cz\in\bbR\,$ characterized 
by $\,(\p\hh-\nh1)K\nnh+\gz=\cz\vt\hh^{\p\hh-\nh1}\nh$, and the following 
triple of real constants:
\begin{equation}\label{trp}
(a,c,r)\,=\,(\p(\p\hh-\nh2)\gz/[4(\p\hh-\nh1)],\,\p\lz/[4(\p+\nh1)],
\,\p\hh\cz/[2(\p^2\nnh-\nh1)])\hh.
\end{equation}
\begin{theorem}\label{rephr}Given a nonconstant function\/ 
$\,f:M\to(0,\infty)\,$ on a Riemannian surface\/ $\,(M\nh,g)$, and an 
Ein\-stein manifold\/ $\,(\varPi\nh,\h)\,$ of dimension\/ $\,\p\ge2$, the 
metric\/ 
$\,f\hh^{4/p}[\hs g+\h\hs]\,$ on\/ $\,M\times\hh\varPi\,$ has harmonic 
curvature if and only if, for the Gauss\-i\-an curvature\/ $\,K\,$ of\/ $\,g$, 
some\/ $\,a,c,r\in\bbR$, and the Ein\-stein constant\/ $\,\gz\,$ of\/ $\,\h$,
\begin{equation}\label{ktr}
\begin{array}{rl}
\mathrm{i)}&
K\nnh=\hn2r(1+1/\p)f\hh^{-\nh2(1-1/\p)}\nh-\gz/(\p\hh-\nh1)\hh,\hskip10pt
\mathrm{ii)}\hskip9.5pt
\p(\p\hh-\nh2)\gz=4(\p\hh-\nh1)a\hh,\\
\mathrm{iii)}&
\Delta\nh f\hs-\,a\nh f\hs
=\,-\hn cf\hh^{1+4/\p}\hs+\,r\nh f\hh^{-\nh1+2/\p}\nh.
\end{array}
\end{equation}
The constant scalar curvature of $\,f\hh^{4/p}[\hs g+\h\hs]\,$ then equals 
$\,4(1+1/\p)\hs c$. Also,
\begin{equation}\label{css}
\mathrm{cases\ (a)\ and\ (b)\ in\ Theorem~\ref{srfbs}\ correspond\ to\ 
}\,r=0\,\mathrm{\ and\ }\,r\ne0\hh.
\end{equation}
\end{theorem}

Here (\ref{css}) is obvious since $\,2(\p^2\nnh-\nh1)\hh r=\p\hh\cz$, cf.\ 
(\ref{trp}), and $\,\p\ge2$.

When $\,M\,$ is compact, and $\,f:M\to(0,\infty)\,$ nonconstant, (\ref{ktr}) 
implies that
\begin{equation}\label{bth}
\begin{array}{rl}
\mathrm{i)}&
c>0\hh,\hskip10pt
\mathrm{ii)}\hskip9.5pt\p\hh-\nh2\,\mathrm{\ and\ }\,a\,\mathrm{\ are\ both\ 
zero,\ or\ both\ positive,}\\
\mathrm{iii)}&
\mathrm{if\ }\,a=0\mathrm{,\ then\ }\,K\hs\mathrm{\ is\ 
nonconstant,\ }\,r>0\mathrm{,\ and\ }\,\p=2\hh,\\
\mathrm{iv)}&
\mathrm{whenever\ }\,r<0\mathrm{,\ one\ has\ }\,\p>2\,\mathrm{\ and\ 
}\,K\nnh<0\,\mathrm{\ everywhere.}
\end{array}
\end{equation}
In fact, (\ref{bth}.ii) follows from (\ref{ktr}.ii), as $\,\gz>0\,$ (see 
Remark~\ref{parts}) and $\,\p\ge2$. Next, one of $\,c\,$ and $\,r\,$ is 
positive: by (\ref{bth}.ii), $\,a\ge0$, so that if we had $\,r\le0\,$ and 
$\,c\le0$, (\ref{ktr}.iii) would make $\,\Delta\nh f\,$ the sum of three
constant or increasing functions of the variable $\,f>0\,$ resulting, via 
Remark~\ref{dlfeq}, in constancy of $\,f\nh$. Nonpositivity of 
$\,c$ would thus lead to positivity of $\,r$, with (\ref{ktr}.iii) 
expressing $\,\Delta\nh f\,$ as the sum of three nonnegative terms and, again, 
contradicting nonconstancy of $\,f\nh$. This yields \hbox{(\ref{bth}.i).} 
If $\,a=0$, we get $\,r>0\,$ (or else $\,\Delta\nh f\,$ would, by 
(\ref{ktr}.iii) and (\ref{bth}.i), be negative), so that (\ref{ktr}.i) and 
(\ref{bth}.ii) yield (\ref{bth}.iii). To prove (\ref{bth}.iv), let $\,r<0$. 
Hence, by (\ref{ktr}.i), $\,K\nnh<0$, as $\,\gz>0\,$ (Remark~\ref{parts}),
while $\,\p>2$, or else (\ref{bth}.ii) with $\,\p=2\,$ and (\ref{bth}.iii)
would give $\,r>0$.
\begin{remark}\label{postv}Positivity of $\,\gz\,$ (or, $\,\lz$) in the 
compact case follows from Remark~\ref{parts} or, respectively, (\ref{bth}.i) 
and (\ref{trp}).
\end{remark}

\section{Vanishing differentials and Hess\-i\-ans\done}\label{vc}
\setcounter{equation}{0}
For a manifold $\,\mathcal{W}\nh$, an interval $\,I\subseteq\bbR$, a 
$\,C^2\nh$ curve $\,I\nh\ni t\mapsto y(t)\in\mathcal{W}\nh$, and a parameter 
$\,c\in I\hs$ such that $\,\dot y(c)=0$, the {\it acceleration\/} vector 
$\,w=\ddot y(c)\in T\hskip-2.7pt_{y(c)}\w\mathcal{W}$ with the components 
$\,w^a\nh=\ddot y\hh^a\nh(c)\,$ in any local coordinates at $\,y(c)\,$ is 
clearly well defined, a co\-or\-di\-nate-free description being: 
$\,d_w\w\phi=\hs d\hs^2[\phi(y(t))]/dt^2\nh$, evaluated at 
$\,t=c$, whenever $\,\phi\,$ is a $\,C^2\nh$ function on a neighborhood of 
$\,y(c)\,$ in $\,\mathcal{W}\nh$. Thus,
\begin{equation}\label{ddy}
\ddot y(c)\,\mathrm{\ equals\ the\ ordinary\ second\ derivative\ of\
}y(t)\,\mathrm{\ at\ }\,t=c
\end{equation}
if $\,\mathcal{W}\hs$ happens to be a $\,C^2\nh$ sub\-man\-i\-fold of a Ba\-nach
space $\,\hat{\mathcal{V}}\nh$, making $\,I\nh\ni t\mapsto y(t)$ a
curve in $\,\hat{\mathcal{V}}\nh$. This is immediate if one
dif\-feo\-mor\-phic\-al\-ly identifies a neighborhood $\,\hat{\mathcal{U}}\hs$
of $\,y(c)\,$ in $\,\hat{\mathcal{V}}\hh$ with 
$\,\,U\nnh\times\hs\hat{\mathcal{U}}'\nh$, for open sub\-sets
$\,\hat{\mathcal{U}}'$ of some Ba\-nach space and $\,\,U\,$ of $\,\rn\nh$, where
$\,n=\dim\mathcal{W}\hs$ and $\,0\in\hat{\mathcal{U}}'\nh$, so as to make
$\,\hat{\mathcal{W}}\nh\cap\hs\hat{\mathcal{U}}\hs$ correspond to 
$\,\,U\nh\times\nh\{0\}$, and then treats the projection
$\,\,U\nnh\times\hs\hat{\mathcal{U}}'\nh\to U\nh$, restricted to
$\,\hat{\mathcal{W}}\nh\cap\hs\hat{\mathcal{U}}\nh=\,U\nh\times\nh\{0\}$, as a
local coordinate system for $\,\hat{\mathcal{W}}\nh$.

Given a $\,C^2\nh$ mapping $\,F:\mathcal{N}\to\mathcal{W}\hs$ between 
manifolds and a point $\,z\in\mathcal{N}$ such that 
$\,d\hskip-.8ptF\hskip-2.3pt_z\w=0$, one defines the {\it Hess\-i\-an\/} of 
$\,F\hh$ at $\,z\,$ to be the symmetric bi\-lin\-e\-ar mapping 
$\,H:T\hskip-2.7pt_z\w\hs\mathcal{N}\nh\times 
T\hskip-2.7pt_z\w\hs\mathcal{N}\nh\to T\hskip-2.7pt_{F(z)}\w\mathcal{W}\hs$ 
characterized by the component formula 
$\,[H(u,v)]^a\nh=H_{\nnh jk}^au^jv^k$ with 
$\,H_{\nnh jk}^a=(\partial\nnh_j\w\hs\partial\nh_k\w F^a)(z)$, whenever 
$\,u,v\in T\hskip-2.7pt_z\w\hs\mathcal{N}\hs$ and $\,x^j$ (or, $\,y^a$) are 
local coordinates in $\,\mathcal{N}\hs$ at $\,z$, or in 
$\,\mathcal{W}\hs$ at $\,F(z)$. An obviously equivalent definition of
$\,H(u,v)$, where symmetry allows us to set $\,u=v$, reads
\begin{equation}\label{itp}
\begin{array}{l}
H(v,v)=\ddot y(c)\,\mathrm{\ for\ the\ curve\ }\,y(t)=F(x(t))\,\mathrm{\ if\
}\,t\mapsto x(t)\\
\mathrm{is\ a\ }\,C^2\nh\mathrm{\ curve\ in\ }\,\,\mathcal{N}\,\mathrm{\
such\ that\ }\,\,x(c)=z\,\,\mathrm{\ and\ }\,\,\dot x(c)=v\hh.
\end{array}
\end{equation}
The acceleration $\,w=\ddot y(c)\,$ in the lines preceding (\ref{ddy})
involves a special case of the Hess\-i\-an; specifically, $\,w=H(u,u)\,$ at
$\,z=c\,$ in $\,\mathcal{N}\nh=I\nh$, the mapping $\,F\,$ and $\,u$ being the
curve and, respectively, $\,1\,$ treated as a vector tangent to $\,I\nh$. 

\section{Fred\-holm differentials and bifurcations\done}\label{ft}
\setcounter{equation}{0}
Suppose that we are given real Ba\-nach spaces
$\,\mathcal{V}\nnh,\hat{\mathcal{V}}\hs$ and a $\,C^k\nh$ mapping $\,\xq$,
$\,1\le k\le\infty$, from a neighborhood of 
$\,0\,$ in $\,\mathcal{V}\hs$ into $\,\hat{\mathcal{V}}\nh$, such that 
$\,\xq(0)=0\,$ and the differential of $\,\xq\,$ at $\,0\,$ is a Fred\-holm 
operator $\,\varPhi=\nh d\xq_0\w:\mathcal{V}\to\hat{\mathcal{V}}\nh$. Thus, 
$\,\mathrm{Ker}\,\varPhi$ and $\,\hat{\mathcal{V}}/\varPhi(\mathcal{V})\,$ 
are fi\-nite-di\-men\-sion\-al, from which closedness of the image 
$\,\varPhi(\mathcal{V})\,$ in $\,\mathcal{V}$ follows 
\cite[p.\ 156]{abramovich-aliprantis}. We fix closed subspaces 
$\,\mathcal{Y}\subseteq\mathcal{V}\hs$ and 
$\,\mathcal{W}\subseteq\hat{\mathcal{V}}\hs$ with 
$\,\mathcal{V}\nh=\nnh\mathcal{Y}\oplus\mathrm{Ker}\,\varPhi$ and 
$\,\hat{\mathcal{V}}\nh=\mathcal{W}\oplus\varPhi(\mathcal{V})$, so that 
$\,\dim\mathcal{W}<\infty$. Due to Ba\-nach's open mapping theorem,
\begin{equation}\label{fdl}
\varPhi=\nh d\xq_0\w\mathrm{\ restricted\ to\ }\,\mathcal{Y}\,\mathrm{\ is\ a\ 
linear\ homeo\-mor\-phism\ }\,\mathcal{Y}\to\varPhi(\mathcal{V})\hh.
\end{equation}
The problem of understanding the pre\-im\-age 
$\,\xq\nnh^{-\nh1}\nh(0)$, clearly contained in 
$\,\xq\nnh^{-\nh1}\nh(\mathcal{W})$, has a {\it local 
fi\-nite-di\-men\-sion\-al reduction}.
\begin{lemma}\label{fdred}
Under the above assumptions, the intersection\/ $\,\mathcal{N}\nh$ of\/ 
$\,\xq\nnh^{-\nh1}\nh(\mathcal{W})\,$ and a suitable neighborhood of\/ 
$\,0\,$ in\/ $\,\mathcal{V}\hs$ forms a\/ $\,C^k\nnh$ manifold of the finite 
dimension\/ 
$\,\dim\hs\mathrm{Ker}\,\varPhi$, having at\/ 
$\,0\,$ the tangent space\/ 
$\,T\hskip-2.7pt_0\w\hs\mathcal{N}\nh=\mathrm{Ker}\,\varPhi$, while\/ 
$\,\xq\hs$ restricted to\/ $\,\mathcal{N}\nh$ constitutes a\/ $\,C^k$
mapping\/ $\,F:\mathcal{N}\to\mathcal{W}\hs$ with\/ $\,F(0)=0\,$ such that\/
$\,d\hskip-.8ptF\nnh_0\w=0\,$ and, if\/ $\,k\ge2$, the Hess\-i\-an\/ $\,H\,$ 
of\/ $\,F\,$ at\/ $\,0\,$ is given by\/ 
$\,H(v,v')=\pi\hs(\nh d\hs[d\xq]_0\w v\hn)v'$ for any\/ 
$\,v,v'\nh\in T\hskip-2.7pt_0\w\hs\mathcal{N}\nh=\mathrm{Ker}\,\varPhi\,$ and 
the projection\/ $\,\pi:\hat{\mathcal{V}}\to\mathcal{W}\hs$ having the 
kernel\/ $\,\varPhi(\mathcal{V})$. 
\end{lemma}
\begin{proof}Let $\,\mathrm{pr}=\mathrm{Id}-\pi\,$ be the projection 
$\,\hat{\mathcal{V}}\to\varPhi(\mathcal{V})\,$ with the kernel 
$\,\mathcal{W}\nh$. Setting $\,S(y,z)=(\mathrm{pr}\hs\xq(y+\nh z),z)$, we 
obtain a $\,C^k$ mapping $\,S\,$ from a neighborhood of $\,0\,$ in 
$\,\mathcal{Y}\times\mathrm{Ker}\,\varPhi\,$ into 
$\,\varPhi(\mathcal{V})\times\mathrm{Ker}\,\varPhi$. The assignment 
$\,(\dot y,\dot z)\mapsto(\varPhi\dot y,\dot z)\,$ represents the differential 
of $\,S\,$ at $\,(0,0)\,$ which -- due to (\ref{fdl}) -- is a linear
homeo\-mor\-phism. Our claim about $\,F:\mathcal{N}\to\mathcal{W}\hs$ now
follows from the inverse mapping theorem: $\,\mathcal{N}\nh$ corresponds via
$\,S\,$ to a neighborhood of $\,(0,0)\,$ in
$\,\{0\}\times\mathrm{Ker}\,\varPhi$, while 
$\,dS_{(0,0)}^{-\nh1}(0,\dot z)=(0,\dot z)$, and $\,d\hskip-.8ptF\nnh_0\w=0\,$ 
since $\,T\hskip-2.7pt_0\w\hs\mathcal{N}\nh=\mathrm{Ker}\,\varPhi
=\mathrm{Ker}\,d\xq_0\w$. To evaluate $\,H$, we choose a curve
$\,t\mapsto y(t)+z(t)
\in\mathcal{N}\nh\subseteq\xq\nnh^{-\nh1}\nh(\mathcal{W})\,$ with
$\,y(t)\in\mathcal{Y}\hh$ and $\,z(t)\in\mathrm{Ker}\,\varPhi$, having at 
$\,t=0\,$ the value $\,0\,$ and velocity
$\,v\in T\hskip-2.7pt_0\w\hs\mathcal{N}\nh=\mathrm{Ker}\,\varPhi$. Thus,
$\,y(0)=z(0)=\dot y(0)=0$ and $\,\dot z(0)=v$, as well as 
$\,\xq(y+\nh z)=\pi\hs\xq(y+\nh z)\,$ for all $\,t$, due to
$\,\mathcal{W}\nh$-val\-ued\-ness of $\,\xq$, where -- from now on -- we write 
$\,y,z,\dot y,\dot z\,$ rather than $\,y(t)$, etc. Applying $\,d/dt$ twice
to the last equality, one gets
$\,d\xq_{y+\nh z}\w(\dot y+\nh\dot z)
=\pi\hs d\xq_{y+\nh z}\w(\dot y+\nh\dot z)\,$ at any $\,t$, and
$\,d\hh[\hh d\xq_{y+\nh z}\w(\dot y+\nh\dot z)]/dt=
\pi\hs(\nh d\hs[d\xq]_0\w v\hn)v\,$ at $\,t=0$, since
$\,d\hh[\hh d\xq_{y+\nh z}\w]/dt=d\hs[d\xq]_{y+\nh z}\w(\dot y+\nh\dot z)$
(with $\,y=z=\dot y=0\,$ and $\,\dot z=v\,$ when $\,t=0$), while
$\,\pi\hs d\xq_0\w(\ddot y+\nh\ddot z)=\pi\hs\varPhi(\ddot y+\nh\ddot z)=0\,$
(note that $\,\varPhi(\mathcal{V})=\hs\mathrm{Ker}\,\pi$).
By (\ref{itp}) and (\ref{ddy}),
$\,H(v,v)=\pi\hs(\nh d\hs[d\xq]_0\w v\hn)v$, and symmetry of $\,H\,$
implies the required formula for $\,H(v,v')$.
\end{proof}

\section{The saddle-point case\done}\label{ov}
\setcounter{equation}{0}
A simple special case of Lemma~\ref{fdred} arises when
\begin{enumerate}
  \def\theenumi{{\rm\roman{enumi}}}
\item $\varPhi(\mathcal{V})\hs$ has the co\-di\-men\-sion $\,1\,$ in 
$\,\hat{\mathcal{V}}$ and, consequently, $\,\dim\hs\mathcal{W}\nh=1$,
\item $k\ge2\,$ and $\,\dim\hs\mathrm{Ker}\,\varPhi=2$, so that
$\,\mathcal{N}\hs$ is a surface,
\item there is an embedded $\,C^1$ curve 
$\,\mathcal{C}\subseteq\mathcal{N}\hs$ with 
$\,0\in\mathcal{C}\subseteq\xq\nnh^{-\nh1}\nh(0)$,
\item we identify $\,\mathcal{W}\hs$ with $\,\bbR$, which turns 
$\,0\in\mathcal{N}\nh$ into a critical point of the $\,C^2$ function 
$\,F:\mathcal{N}\to\bbR\,$ on the surface $\,\mathcal{N}\nnh$, having 
$\,F(0)=0$,
\item $H(v,v')\ne0\,$ for the Hess\-i\-an $\,H\,$ of $\,F\,$ at $\,0$, some 
vector $\,v\,$ tangent to the curve $\,\mathcal{C}\,$ at $\,0$, and some 
$\,v'\nh\in T\hskip-2.7pt_0\w\hs\mathcal{N}\nh=\mathrm{Ker}\,\varPhi$.
\end{enumerate}
Then $\,H\,$ is indefinite. Namely, $\,H\nh\ne0$, while $\,H(v,v)=0\,$ since 
$\,\xq=0\,$ along $\,\mathcal{C}\nh$, and so $\,H\,$ cannot be definite (or 
sem\-i\-def\-i\-nite), or else we would have $\,v=0\,$ (or $\,H(v,v')=0$). As
a result, $\,F\,$ has a saddle point at $\,0$, and 
\begin{equation}\label{int}
\begin{array}{l}
\mathrm{the\hs\ intersection\hs\ of\hs\ 
}\,\,\xq\nnh^{-\nh1}\nh(0)\,\,\mathrm{\ with\hs\ a\hs\ neighborhood\hs\ 
of\ }\,\hs0\hs\,\mathrm{\ in\ }\,\mathcal{N}\\
\mathrm{is\ the\ union\ of\ two\ embedded\ curves\ intersecting\ 
transversal}\hyp\\
\mathrm{ly\ at\ }\,\,\hs0\hs\in\mathcal{N}\,\mathrm{\ and\ having\ no\ 
other\ points\ in\ common;\ one\ of}\\
\mathrm{these\ curves\ is\ contained\ in\ our\ 
}\,\hh\mathcal{C}\nh\mathrm{,\ the\ other\ has\ the\ tangent}\\
\mathrm{line\ }\,\bbR w\,\mathrm{\ at\ }\,0\mathrm{,\ for\ some\ 
}\,w\in\mathrm{Ker}\,\varPhi\smallsetminus\bbR v\,\mathrm{\ with\ 
}\,H(w,w)=0\hh.
\end{array}
\end{equation}
We will refer to these two curves, respectively, as
\begin{equation}\label{rfr}
\mathrm{the\ curve\ of\ trivial\ solutions\ (contained\ in\ 
}\,\mathcal{C}\mathrm{),\ and\ the\ bifurcating\ branch.}
\end{equation}
Next, given real Ba\-nach spaces $\,\mathcal{V}\nnh,\hat{\mathcal{V}}\hs$ and a
mapping $\,\xq\,$ of class $\,C^k\nnh$, $\,2\le k\le\infty$, from a 
neighborhood of $\,0\,$ in $\,\mathcal{V}\hs$ into $\,\hat{\mathcal{V}}\hs$ 
with $\,\xq(0)=0$, 
suppose that
\begin{enumerate}
  \def\theenumi{{\rm\alph{enumi}}}
\item $\mathcal{V}\nh=\mathcal{V}'\nnh\times\bbR$, for a Ba\-nach space  
$\,\mathcal{V}'\nh$,
\item $\xq\nh^t(0)=0\,$ for all $\,t\,$ near $\,0\,$ in $\,\bbR$, where we set
$\,\xq\nh^t(x)=\xq(x,t)$,
\item $d\xq^0_0$ (the differential of $\,\xq\nh^0$ at $\,0\in\mathcal{V}'$) is
a Fred\-holm operator,
\item $\dim\,\mathrm{Ker}\,d\xq^0_0
=\dim\,[\hh\hat{\mathcal{V}}\nh/\nh d\xq^0_0(\mathcal{V}')]=\hs1$,
\item $d\xq^0_0(\mathcal{V}')\hs\cap\hs d\nh\dot\xq^0_0(\mathrm{Ker}\,d\xq^0_0)
  =\{0\}\ne\hs d\nh\dot\xq^0_0(\mathrm{Ker}\,d\xq^0_0)$, with
  $\,\dot\xq\nh^t\nh=d\xq\nh^t\nnh\hn/\nh dt$. 
\end{enumerate}
\begin{lemma}\label{vspec}
Under the assumptions\/ {\rm(a)} -- {\rm(e)}, the hypotheses of
Lemma\/~{\rm\ref{fdred}} along with conditions\/ {\rm(i)} -- {\rm(v)} above
are all satisfied, and hence so are their conclusions, including\/
{\rm(\ref{int})}, while\/ 
$\,T\hskip-2.7pt_{(0,0)}\w\hs\mathcal{N}\nh
=\mathrm{Ker}\,d\xq^0_0\times\bbR$. For the Hess\-i\-an\/ $\,H\,$ of\/ $\,F\,$
at\/ $\,(0,0)$ and any vectors\/ 
$\,v,v\nh'\in\mathrm{Ker}\,d\xq^0_0\times\bbR\,$ of the form\/ 
$\,v=(0,1)\,$ and\/ $\,v'\nh=(u,0)$, where\/ $\,u\in\mathrm{Ker}\,d\xq^0_0$, 
one has\/ $\,H(v,v')=d\nh\dot\xq^0_0u\nh$. The curve\/ $\,\mathcal{C}\,$ of
condition\/ {\rm(iii)} is a neighborhood of\/ $\,(0,0)\,$ in\/
$\,\{0\}\times\bbR$.
\end{lemma}
\begin{proof}The hypotheses of Lemma~\ref{fdred} easily follow from (a) --
(e), and so do (i) -- (iv): the Fred\-holm property of
$\,\varPhi=\nh d\xq_{(0,0)}\w$, with the dimensions required in (i) -- (ii), 
is obvious since $\,\varPhi\,$ has the kernel 
$\,\mathrm{Ker}\,d\xq^0_0\times\bbR\,$ and the image
$\,d\xq^0_0(\mathcal{V}')$. Finally, for
$\,v,v\nh'\in T\hskip-2.7pt_{(0,0)}\w\hs\mathcal{N}\nh
=\mathrm{Ker}\,\varPhi=\mathrm{Ker}\,d\xq^0_0\times\bbR\,$ as in the statement
of the lemma, with a nonzero vector $\,u\in\mathrm{Ker}\,d\xq^0_0$, the
formula $\,H(v,v')=\pi\hs(\nh d\hs[d\xq]_0\w v\hn)v'$ of Lemma~\ref{fdred}
reads $\,H(v,v')=\pi\hs d\nh\dot\xq^0_0u\,$ while, by (e),
$\,d\nh\dot\xq^0_0u\notin d\xq^0_0(\mathcal{V}')$. The relation 
$\,d\xq^0_0(\mathcal{V}')=\varPhi(\mathcal{V})=\hs\mathrm{Ker}\,\pi\,$ now
yields $\,H(v,v')\ne0$, proving (v).
\end{proof}

\section{Nonconstant Gauss\-i\-an curvature: part one\done}\label{nc}
\setcounter{equation}{0}
Compact warped products $\,(M\times\hh\varPi\nh,\bg)\,$ with harmonic 
curvature, nonconstant warping functions, and two\hs-di\-men\-sion\-al bases 
$\,(M\nh,g)\,$ represent two separate cases, (a) and (b) in 
Theorem~\ref{srfbs}. Case (b), discussed here, amounts, by 
Theorem~\ref{rephr}, to having $\,\bg=f\hh^{4/p}[\hs g+\h\hs]\,$ for an 
Ein\-stein metric $\,\h\,$ with some Ein\-stein constant $\,\gz\,$ and a 
nonconstant function $\,f:M\to(0,\infty)\,$ satisfying (\ref{ktr}), that is,
\begin{equation}\label{dlk}
\mathrm{a)}\hskip7pt
\Delta\nh f\nh=\varOmega(f)\hh,\hskip12pt
\mathrm{b)}\hskip7ptK\nnh
=\hn2r(1+1/\p)f\hh^{-\nh2(1-1/\p)}\nh-\gz/(\p\hh-\nh1)\hh,
\end{equation}
$K\hs$ and $\,\varOmega\,$ being the Gauss\-i\-an curvature of $\,g\,$ and the 
function on $\,(0,\infty)\,$ given by
\begin{equation}\label{omf}
\varOmega(f)=a\nh f\nh-\hn cf\hh^{1+4/\p}\nh
+r\nh f\hh^{-\nh1+2/\p}\hskip9pt\mathrm{with}\hskip7pta
=\p(\p\hh-\nh2)\gz/[4(\p\hh-\nh1)]\hh.
\end{equation}
Here $\,\p\ge2\,$ is the dimension of the fibre, $\,c,r\in\bbR$, and 
$\,r\ne0$, cf.\ (\ref{css}), while $\,K$ must be nonconstant due to 
nonconstancy of $\,f\hs$ and (\ref{dlk}.b).

In the next section we will use the bifurcation method of Lemma~\ref{vspec} to 
prove the existence of Riemannian metrics $\,g\,$ on compact surfaces $\,M\,$ 
admitting nonconstant functions $\,f:M\to(0,\infty)\,$ with (\ref{dlk}) -- 
(\ref{omf}). Such $\,g\,$ will arise from con\-for\-mal changes of the form 
$\,g=e^{2x}\hg\nh$, where the metric $\,\hg\,$ on $\,M\,$ has constant 
Gauss\-i\-an curvature $\,\hat K\nh$, and $\,x:M\to\bbR$. However, rather than 
being smooth, $\,x\,$ is only required to lie in a suitable $\,L\nh^2$ 
So\-bo\-lev space, chosen so as to ensure 
$\,C^4\nnh$-dif\-fer\-en\-ti\-a\-bil\-i\-ty of $\,x$.

Our approach uses a fixed choice of the data 
$\,M\nh,\hg,\hat K\nh,\p,\q,r,\lambda\,$ consisting of a compact Riemannian 
surface $\,(M\nh,\hg)\,$ of constant Gauss\-i\-an curvature 
$\,\hat K\nnh\ne0$, integers $\,\p\ge2\,$ and $\,\q\ge6$, a real parameter 
$\,r\ne0$, and a suitable eigen\-val\-ue $\,\lambda$ of $\,-\nh\hat\Delta$, 
for the $\,\hg\nh$-La\-plac\-i\-an $\,\hat\Delta$. The Gauss\hh-\nh Bonnet 
theorem and (\ref{bth}.iv) make it necessary to assume that
\begin{equation}\label{gbt}
\mathrm{if\ }\,\,r<0\mathrm{,\ \ \ then\ }\,\,\p>2\,\,\mathrm{\ and\ 
}\,\,\hat K\nnh<0\hh.
\end{equation}
By a {\it solution of\/} (\ref{dlk}) we then mean a quadruple 
$\,(x,f,\gz,c)\,$ formed by a $\,C^4$ function $\,x:M\to\bbR$, a $\,C^2$ 
function $\,f:M\to(0,\infty)$, and constants $\,\gz,c\in\bbR$ such that 
(\ref{dlk}), with (\ref{omf}), holds for the Gauss\-i\-an curvature $\,K\hs$ 
of the $\,C^4$ metric $\,g=e^{2x}\hg\,$ on $\,M\,$ and the 
$\,g\hh$-\nh La\-plac\-i\-an $\,\Delta\,$ (the objects
$\,M\nh,\hg,\hat K\nh,\p,r$ still being fixed).

In contrast with the lines surrounding (\ref{dlk}) -- (\ref{omf}), $\,f\,$ and 
$\,K\hs$ are this time allowed to be constant: in fact, there exist {\it
trivial solutions\/} of (\ref{dlk}), namely,
$\,(x,f\nh,\gz,c)$ having $\,x=0$, a constant $\,f\nh>0$, and
$\,\gz,c\in\bbR\,$ chosen so as to yield (\ref{dlk}) -- (\ref{omf}) with
$\,K\nnh=\hat K\hs$ and $\,\varOmega(f)=0$, that is,
\begin{equation}\label{eep}
\begin{array}{l}
\gz\,=\,(\p\hh-\nh1)[\hn2r(1+1/\p)f\hh^{-\nh2(1-1/\p)}\hs-\,\hat K]\hh,
\hskip10pt\mathrm{and}\\
c=a\nh f\hh^{-4/\p}\nh+r\nh f\hh^{-\nh2-2/\p}\hskip7pt\mathrm{for\ 
}\,a=\p(\p\hh-\nh2)\gz/[4(\p\hh-\nh1)]\hh.
\end{array}
\end{equation}
This curve of trivial solutions is pa\-ram\-e\-triz\-ed by $\,f\in(0,\infty)$, 
and some of them can be deformed to bifurcating branches of solutions with 
nonconstant $\,f\,$ and 
$\,K\nh$. There are obstructions to such a deformation, in the form of {\it
  three positivity conditions\/} imposed on the constant $\,f\nh>0$. The first
two reflect the fact that nonconstancy of $\,f\hs$ gives
$\,\gz,c\in(0,\infty)$, cf.\ Remark~\ref{parts} and (\ref{bth}.i), while -- in
trivial solutions -- $\,\gz,c\,$ depend on $\,f\,$ via (\ref{eep}). The third
condition arises since a bifurcation can only occur at $\,f\,$ if the value of
$\,f\,$ is quite specifically related to a nonzero (and hence {\it positive}) 
eigen\-val\-ue $\,\lambda\,$ of $\,-\nh\hat\Delta$, for the 
$\,\hg\nh$-La\-plac\-i\-an $\,\hat\Delta$. See formula (\ref{lbd}.i) below.

It is convenient to replace the parameter $\,f\in(0,\infty)\,$ mentioned above 
with the positive real variable $\,\theta=f\hh^{-\nh2(1-1/\p)}\nh$. For the 
trivial solution $\,(x,f\nh,\gz,c)\,$ of (\ref{dlk}) corresponding to 
$\,\theta\,$ one then has $\,x=0\,$ and $\,\,f\nh=\theta\hh^{\p/(2-2\p)}\nh$, 
whereas (\ref{eep}) reads
\begin{equation}\label{pos}
\gz=2(\p\hh-\nh1/\p)\hs r\hh\theta-(\p\hh-\nh1)\hat K\nh,\hskip12pt
4\hh c/\p
=[2(\p\hh-\nh1)\hs r\hh\theta-(\p\hh-\nh2)\hat K]\hs\theta\hh^{2/(\p\hh-\nh1)}\nh.
\end{equation}
In terms of $\,\theta$, the relation between $\,f\,$ and the eigen\-val\-ue 
$\,\lambda\,$ of $\,-\nh\hat\Delta\,$ takes the form
\begin{equation}\label{lbd}
\mathrm{i)}\hskip7pt\lambda\,=\,2(\p\hh-\nh1/\p)\hs r\hh\theta\,
-\,(\p\hh-\nh2)\hat K\nh,\hskip12pt\mathrm{that\ 
is,}\hskip9pt\mathrm{ii)}\hskip7pt\lambda\,=\,\gz\,+\,\hat K\nh,
\end{equation}
justified later by (\ref{bda}) and Lemma~\ref{dlddl}. If 
$\,\theta\in(0,\infty)$, simultaneous positivity of the three constants 
$\,\gz,c,\lambda\,$ in (\ref{pos}) -- (\ref{lbd}.i) clearly amounts to
\begin{equation}\label{smp}
2(\p^2\nnh-\nh1)\hs r\hh\theta\,
>\,\mathrm{max}\,\{\hs\p(\p\hh-\nh1)\hat K\nh,(\p+\hn1)(\p\hh-\nh2)\hat K\nh,
\p(\p\hh-\nh2)\hat K\}\hh.
\end{equation}
With $\,M\nh,\hg,\hat K\nh,\p,r\,$ still fixed, let 
$\,I\nnh_r\w\subseteq(0,\infty)\,$ be the open interval defined by
\begin{equation}\label{ire}
  I\nnh_r\w=(\theta\nnh_+\w,\infty)\mathrm{,\ when\ }\,r>0\mathrm{,\ or\ 
}\,I\nnh_r\w=(0,\theta\nnh_-\w)\mathrm{,\ for\ }\,r<0\hh,
\end{equation}
where $\,\theta\nnh_+\w=\mathrm{max}\,\{\p\hat K/[2(\p+\hn1)\hs r],0\}$ and
$\,\theta\nnh_-\w=\p(\p\hh-\nh2)\hat K/[2(\p^2\nnh-\nh1)\hs r]$. Note that 
$\,\theta\nnh_-\w>0\,$ if $\,r<0$, due to (\ref{gbt}), while
\begin{equation}\label{ppm}
\p(\p\hh-\nh2)\,\le\,(\p+\hn1)(\p\hh-\nh2)\,
<\,\p(\p\hh-\nh1)\hskip9pt\mathrm{whenever\ }\,\p\ge2\hh.
\end{equation}
Our three positivity conditions mean precisely that 
$\,\theta\in I\nnh_r\w$. Namely, we have
\begin{lemma}\label{simul}
The interval\/ $\,I\nnh_r\w$ is the set of all\/ 
$\,\theta\in(0,\infty)\,$ for which the three expressions\/ $\,\gz,c\,$ and\/ 
$\,\lambda$, given by\/ {\rm(\ref{pos})} -- {\rm(\ref{lbd}.i)}, are 
simultaneously positive.
\end{lemma}
\begin{proof}Depending on whether $\,r>0\,$ and $\,\hat K\nh>0\,$ (or, 
$\,r>0\,$ and $\,\hat K\nh<0$ or, respectively, $\,r<0$, so that 
(\ref{gbt}) gives $\,\hat K\nnh<0$), condition (\ref{smp}) imposed on 
$\,\theta\in(0,\infty)\,$ reads, by (\ref{ppm}), 
$\,\theta>\p\hat K/[2(\p+\hn1)\hs r]$, or $\,\theta>0\,$ or, respectively, 
$\,\theta<\,\p(\p\hh-\nh2)\hat K/[2(\p^2\nnh-\nh1)\hs r]$, as required.
\end{proof}
\begin{remark}\label{sblev}Given a compact Riemannian manifold $\,(M\nh,g)\,$
of any dimension $\,\m\,$ and an open interval $\,I\hh\subseteq\bbR$, the
So\-bo\-lev embedding theorem implies that, if $\,\q>\m$, the So\-bo\-lev space
$\,L_{\hn\q}^{\nh2}(M\nh,\bbR)\,$ of functions with $\,\q\,$ derivatives in
$\,L\nh^2$ can be turned into a Ba\-nach algebra, while the $\,I\nh$-val\-ued
functions in $\,L_{\hn\q}^{\nh2}(M\nh,\bbR)\,$ form an open subset
$\,L_{\hn\q}^{\nh2}(M\nh,I)\,$ of $\,L_{\hn\q}^{\nh2}(M\nh,\bbR)$. On the
other hand, for any Ba\-nach algebra $\,\mathcal{A}$, convergent power series
define $\,\mathcal{A}$-val\-ued $\,C^\infty$ functions on open subsets of
$\,\mathcal{A}$. Applied to $\,\mathcal{A}=L_{\hn\q}^{\nh2}(M\nh,\bbR)$, this
yields $\,\mathcal{A}$-val\-ued\-ness and
$\,C^\infty\nnh$-dif\-fer\-en\-ti\-a\-bil\-i\-ty of the mapping
$\,L_{\hn\q}^{\nh2}(M\nh,I)\ni x\mapsto\varphi\circ x$, whenever the
function $\,\varphi:I\to\bbR\,$ is real-an\-a\-lyt\-ic.
\end{remark}

\section{Nonconstant Gauss\-i\-an curvature: part two\done}\label{cb}
\setcounter{equation}{0}
We now proceed to construct metrics on closed surfaces realizing case (b) in 
Theorem~\ref{srfbs}. Curves of such metrics $\,g$, emanating from a 
con\-stant-cur\-va\-ture metric $\,\hg\nh$, will arise via the bifurcation 
argument of Lemma~\ref{vspec}. As outlined in Section~\ref{nc}, the 
construction uses a fixed septuple 
$\,M\nh,\hg,\hat K\nh,\p,\q,r,\lambda\,$ formed by
\begin{enumerate}
  \def\theenumi{{\rm\roman{enumi}}}
\item a closed Riemannian surface $\,(M\nh,\hg)\,$ of nonzero constant 
Gauss\-i\-an curvature $\,\hat K\nh$, along with integers $\,\p\ge2\,$ and 
$\,\q\ge6$,
\item a real parameter $\,r\ne0$, satisfying (\ref{gbt}): $\,r>0\,$ unless 
$\,\hat K\nnh<0\,$ and $\,\p>2$,
\item a constant $\,\lambda\in(0,\infty)\,$ such that, for the 
$\,\hg\nh$-La\-plac\-i\-an $\,\hat\Delta$,
\begin{enumerate}
  \def\theenumi{{\rm\alph{enumi}}}
\item $\lambda=2l(2l+1)\hat K\nh$, where $\,l\,$ is a positive integer, 
if $\,\hat K\nh>0$,
\item $[\lambda+(\p\hh-\nh2)\hat K]\hs r>0\,$ and 
$\,\dim\hs\mathrm{Ker}\hskip1.7pt(\hat\Delta+\lambda)=1$, when 
$\,\hat K\nnh<0$.
\end{enumerate}
\end{enumerate}
In both cases (iii.a) -- (iii.b), $\,\lambda\,$ is a positive eigen\-val\-ue
of $\,-\nh\hat\Delta\,$ (see Section~\ref{cn}).

For $\,I\nnh_r\w\subseteq(0,\infty)\,$ as in (\ref{ire}), let $\,\tz\in\bbR\,$ 
and $\,I\subseteq\bbR\,$ be given by
\begin{equation}\label{dep}
\tz\,=\,\p\hs[\lambda+(\p\hh-\nh2)\hat K]/[2(\p^2\nnh-\nh1)\hs r]\hh,\hskip16pt
I\hs=\,\{t\in\bbR:\tz+t\in I\nnh_r\w\}\hh.
\end{equation}
\begin{lemma}\label{dinir}
Under the assumptions\/ {\rm(i)} -- {\rm(iii)}, $\,\tz\in I\nnh_r\w$ and\/ 
$\,I\hs$ is an open interval containing\/ $\,0$.
\end{lemma}
\begin{proof}The condition $\,\tz\in I\nnh_r\w$ reads 
$\,r\hh\tz>r\hh\theta\nnh_+\w=\mathrm{max}\,\{\p\hat K/[2(\p+\hn1)],0\}\,$ 
when $\,r>0$, and 
$\,0>r\hh\tz>r\hh\theta\nnh_-\w=\p(\p\hh-\nh2)\hat K/[2(\p^2\nnh-\nh1)]\,$ if 
$\,r<0$. Thus, $\,\tz\in I\nnh_r\w$ by (iii), since $\,\p\ge2$, and 
(\ref{dep}) yields 
$\,2(\p^2\nnh-\nh1)\hs r\hh\tz/\p=\lambda+(\p\hh-\nh2)\hat K\nh$.
\end{proof}
For our fixed septuple $\,M\nh,\hg,\hat K\nh,\p,\q,r,\lambda$, any given 
$\,t\in I\nh$, a function $\,x:M\to\bbR$ having some further properties, 
named in the paragraph following (\ref{ltz}), and $\,\tz\,$ as in (\ref{dep}), 
we let $\,\gz,c,K\nh,\Delta,\varOmega\,$ and $\,f\,$ denote the constants in 
(\ref{pos}) with $\,\theta=\tz+t$, the Gauss\-i\-an curvature of the metric 
$\,g=e^{2x}\hg\nh$, the $\,g$-La\-plac\-i\-an, the function (\ref{omf}), and 
$\,f\,$ characterized by (\ref{dlk}.b), that is, by 
$\,K\nnh=\hn2r(1+1/\p)f\hh^{-\nh2(1-1/\p)}\nh-\gz/(\p\hh-\nh1)$. Using 
$\,K\nh,\Delta,\varOmega\,$ and $\,f\,$ depending on $\,t,x\,$ as described 
here, we define $\,\xq\nh^t(x)=\xq(x,t)\,$ to be 
$\,\Delta\nh f\nh-\varOmega(f)$. Explicitly,
\begin{equation}\label{ltx}
\begin{array}{l}
\xq\nh^t(x)\,=\,\Delta\nh f\hs-\,a\nh f\hs+\,cf\hh^{1+4/\p}\hs
-\,r\nh f\hh^{-\nh1+2/\p}\nh,\hskip9pt\mathrm{where}\\
a=\p(\p\hh-\nh2)\gz/[4(\p\hh-\nh1)]\hskip6pt\mathrm{for}\hskip6pt
\gz=2(\p-\hn1/\p)\hs r\hh(\tz+t)-(\p\hh-\nh1)\hat K\nh,\\
c\,=\,\p\hh[2(\p\hh-\nh1)\hs r\hh(\tz+t)-(\p\hh-\nh2)\hat K]\hs
(\tz+t)^{2/(\p\hh-\nh1)}\nnh/4\hh,\\
f\,=\,\,[2r(1+1/\p)]^{-\p/(2-2\p)}[K\hh+\,\gz/(\p\hh-\nh1)]^{\p/(2-2\p)}\nh,
\hskip9pt\mathrm{and}\\
g\,=\,e^{2x}\hg\hs,\hskip12pt\Delta\,
=\,e^{-\nh2x}\nh\hat\Delta\hs,\hskip12ptK\hs
=\,e^{-\nh2x}(\hat K-\hat\Delta x)\hh,\hskip9pt\mathrm{cf.\ 
Remark~\ref{sccrv}.}
\end{array}
\end{equation}
Since (\ref{ltx}) easily shows that, whenever $\,t\in I\nh$,
\begin{equation}\label{kaf}
K\hs\mathrm{\ and\ }\,f\,\mathrm{\ have,\ at\ }\,x=0\,\mathrm{\ and\ 
}\,t,\mathrm{\ the\ values\ }\,\hat K\hs\mathrm{\ and\ 
}\,(\tz+t)^{\p/(2-2\p)}\nh,
\end{equation}
relations (\ref{ltx}) easily yield
\begin{equation}\label{ltz}
\xq\nh^t(0)\,=\,0\,\,\mathrm{\ for\ all\ }\,\,t\in I\nh.
\end{equation}
As for $\,x$, we require that it be close to $\,0\,$ in a sub\-space 
$\,\mathcal{V}'$ -- described in the lines preceding (\ref{dko}) -- of the 
So\-bo\-lev space $\,L_{\hn\q}^{\nh2}(M\nh,\bbR)$, with $\,\q\ge6\,$
derivatives in $\,L\nh^2\nh$. The So\-bo\-lev embedding theorem then
guarantees $\,C\hh^{\q-2}$ differentiability of $\,x$, while its closeness to
$\,0\,$ is meant to ensure positivity of $\,f\hs$ via that of
$\,\tz+t\in I\nnh_r\w$ in (\ref{kaf}), the latter due to the definition of
$\,I\nh$, cf.\ (\ref{dep}), and the inclusion
$\,I\nnh_r\w\subseteq(0,\infty)$.

Our data $\,\hat K\nh,\p,r,\tz\,$ are constants, while $\,\gz\,$ and $\,c\,$
depend only on $\,t\,$ (not on $\,x$), $\,K\hs$ only on $\,x$, and $\,f\,$ on
both $\,x,t$. Therefore, by (\ref{ltx}) and (\ref{kaf}), for the differentials
of $\,\hat K\nh,\p,r,\tz,\gz,c,K\hs$ and $\,f\,$ with respect to the variable 
$\,x\in\mathcal{V}'\nh$, at $\,x=0\,$ and any $\,t\in I\nh$, one has
\begin{equation}\label{dif}
\begin{array}{l}
d\p_0\w\,=\hs\,d\hs r\nh_0\w\,=\hs\,d\hat K\nnh_0\w\,=\hs\,d\tz\hn_0\w\,
=\hs\,d\gz\hn_0\w\,=\hs\,d\hh c_0\w\,=\hs\,0\hh,\\
dK\nnh_0\w=-4\hs r(1-1/\p^2)(\tz+t)^{(2-3\p)/(2-2\p)}\df\nnh_0\w
=-\hn(\hat\Delta+\nh2\hat K)\hh,
\end{array}
\end{equation}
$2\hat K\,$ denoting here $\,2\hat K\,$ times the identity. From (\ref{dep}),
\begin{equation}\label{bda}
\lambda\,=\,2(\p\hh-\nh1/\p)\hs r\hh\tz\,-\,(\p\hh-\nh2)\hat K\,
\in\,(0,\infty)\hh,
\end{equation}
which is also the value of $\,\lambda\,$ in Lemma~\ref{simul} for 
$\,\theta=\tz$.
\begin{lemma}\label{dlddl}
With notations of Section\/~{\rm\ref{ov}}, 
$\,4\hs r(1-1/\p^2)(\tz+t)^{1\nh-\p/(2-2\p)}\hs d\xq\nh^t_0
=\hbox{$[\hat\Delta+\lambda+2(\p\hh-\nh1/\p)\hs rt]
(\hat\Delta+\nh2\hat K)$}$, as well as\/ 
$\,8\hs r(\p\hh-\nh1)(1-1/\p^2)\tz^{2-\p/(2-2\p)}\hs d\nh\dot\xq^0_0
=\hbox{$[(2-3\p)\hat\Delta-\p\lambda+2(\p\hh-\nh1)(\p\hh-\nh2)\hat K]
(\hat\Delta+\nh2\hat K)$}$, at any\/ $\,t\in I\nh$, or\/ $\,t=0$, 
and\/ $\,x=0$. 
\end{lemma}
\begin{proof}By (\ref{ltx}), 
$\,d\hs[\Delta\nh f]_0\w=\hat\Delta\hs\df\nnh_0\w$ and 
$\,d\hs[\varOmega(f)]_0\w=\varOmega'\nnh(f)\,\df\nnh_0\w$, with 
$\,\varOmega'\nnh(f)$ denoting the derivative $\,d\varOmega\nh/\hn df\,$ at 
$\,f\nh=(\tz+t)^{\p/(2-2\p)}$ (the value of $\,f\,$ for 
$\,x=0$, which is a constant function on $\,M\nh$, depending on $\,t$). From 
(\ref{omf}) and (\ref{ltx}), 
$\,\varOmega'\nnh(f)=a-(1+4/\p)\hh cf\hh^{4/\p}\nh
-(1-2/\p)\hs r\nh f\hh^{-\nh2+2/\p}\nh
=-\hh[2(\p\hh-\nh1/\p)\hs r\hh(\tz+t)-(\p\hh-\nh2)\hat K]$ which, by 
(\ref{bda}), equals $\,-\nh2(\p\hh-\nh1/\p)\hs rt-\lambda$. This yields 
$\,d\xq\nh^t_0=d\hs[\Delta\nh f]_0\w-d\hs[\varOmega(f)]_0\w
=[\hat\Delta+\lambda+2(\p\hh-\nh1/\p)\hs rt]\,\df\nnh_0\w$, cf.\ (\ref{ltx}),
and the last line of (\ref{dif}) implies the first equality; applying
$\,d/dt\,$ to it and using (\ref{bda}), we obtain the second one.
\end{proof}
To use Lemma~\ref{vspec}, we fix 
$\,M\nh,\hg,\hat K\nh,\p,\q,r,\lambda,\tz\,$ as in (i) -- (iii) and
(\ref{dep}), along with specific vector sub\-spaces $\,\mathcal{V}'$ of 
$\,L_{\hn\q}^{\nh2}(M\nh,\bbR)\,$ and $\,\hat{\mathcal{V}}$ of
$\,L_{\hn\q-4}^{\nh2}(M\nh,\bbR)\,$ such that
\begin{enumerate}
  \def\theenumi{{\rm\roman{enumi}}}
\item[(iv)] $\hat{\mathcal{V}}$ contains
$\hs\hat{\mathcal{V}}'\nh=(\hat\Delta+\lambda)(\mathcal{V}')\,$ and all 
$\,\xq\nh^t(\mathcal{V}')$, $\,t\in I\nh$, while
$\,\dim\,[\hh\hat{\mathcal{V}}\nh/\hh\hat{\mathcal{V}}']=\hs1$.
\end{enumerate}
Here is how we select $\,\mathcal{V}'$ and $\,\hat{\mathcal{V}}\nh$. For
$\,\hat K\nnh<0$, we set $\,\mathcal{V}'\nh=L_{\hn\q}^{\nh2}(M\nh,\bbR)\,$ and 
$\,\hat{\mathcal{V}}\nh=L_{\hn\q-4}^{\nh2}(M\nh,\bbR)$. If $\,\hat K\nh>0$, we
fix a nontrivial $\,\hg\hn$-iso\-met\-ric action of the circle group $\,S^1$
on $\,M=\bbR\mathrm{P}^2$ (or, $\,M=S^2$) and let the sub\-spaces \
$\,\mathcal{V}'\nh,\hat{\mathcal{V}}$ of $\,L_{\hn\q}^{\nh2}(M\nh,\bbR)\,$ and
$\,L_{\hn\q-4}^{\nh2}(M\nh,\bbR)\,$ consist of all $\,S^1\nh$-in\-var\-i\-ant
functions required, in the case of $\,M=S^2\nh$, to be also invariant under 
the antipodal isometry. In both cases one has (iv), since
\begin{equation}\label{dko}
\dim\hskip1.2pt[\hh\mathcal{V}'\nh\cap\hs\mathrm{Ker}\hskip1.7pt(\hat\Delta
+\lambda)]\,=\,1\hh,
\end{equation}
due to (iii.b) or, respectively, (iii.a) combined with Remark~\ref{eignv}.
\begin{lemma}\label{cndae}
Conditions\/ {\rm(a)} -- {\rm(e)} of Section\/~{\rm\ref{ov}} are all satisfied 
by\/ $\,\mathcal{V}'\nnh,\hat{\mathcal{V}}\hs$ and\/ $\,\xq\,$ chosen as
above, with\/ $\,\mathcal{V}\nh=\mathcal{V}'\nnh\times\bbR\,$ and\/
$\,k=\infty$, while\/
$\,\mathrm{Ker}\,d\xq^0_0
\subseteq\hs\mathrm{Ker}\hskip1.7pt(\hat\Delta+\lambda)$. 
\end{lemma}
\begin{proof}First, $\,C^\infty\nnh$-dif\-fer\-en\-ti\-a\-bil\-i\-ty of
$\,\xq\,$ and (a) -- (b) are obvious from Remark~\ref{sblev} and,
respectively, (\ref{ltz}). (The former also holds for a more general reason:
one can derive it from Ne\-myts\-ky's theorem 
\cite[Section 10.3.4]{renardy-rogers}, without invoking 
real\hh-an\-a\-lyt\-ic\-i\-ty.) Next, we have (c) -- (d). Namely, due to
(iii-a) and Remark~\ref{eignv}, $\,\hat\Delta+\nh2\hat K\hs$ is injective on
$\,\mathcal{V}'\nh$. Thus, 
Lemma~\ref{dlddl} for $\,t=0\,$ gives $\,\mathrm{Ker}\,d\xq^0_0
=\mathcal{V}'\nh\cap\mathrm{Ker}\hskip1.7pt(\hat\Delta+\lambda)\,$ and 
$\,d\xq^0_0(\mathcal{V}')=(\hat\Delta+\lambda)(\mathcal{V}')$, while
(\ref{dko}) and (iv) imply one\hs-di\-men\-sion\-al\-i\-ty of both spaces in
(d). Finally, (e) follows since the restriction of the factor
$\,(2-3\p)\hat\Delta-\p\lambda+2(\p\hh-\nh1)(\p\hh-\nh2)\hat K\hs$ in the last
equality of Lemma~\ref{dlddl} to
$\,\mathrm{Ker}\hskip1.7pt(\hat\Delta+\lambda)\,$ equals the identity times 
$\,2(\p\hh-\nh1)[\lambda+(\p\hh-\nh2)\hat K]$, which is nonzero as a
consequence of (iii).
\end{proof}
Lemma~\ref{cndae} allows us to apply Lemma~\ref{vspec} to our 
$\,\mathcal{V}'\nnh,\hat{\mathcal{V}}\nh,\xq\,$ and 
$\,\mathcal{V}\nh=\mathcal{V}'\nnh\times\bbR$, arising from a fixed septuple 
$\,M\nh,\hg,\hat K\nh,\p,\q,r,\lambda$, which yields
(i) -- (v) of Section~\ref{ov}, along with (\ref{int}). We define the
$\,\lambda${\it-branch\/} corresponding to these data to be the set of all
$\,g=e^{2x}\hg$, where $\,(x,t)\in\mathcal{V}'\nh\times\bbR\,$ ranges over the
bifurcating branch of solutions introduced in (\ref{rfr}). The
$\,\lambda$-branches, associated with all positive eigen\-val\-ues
$\,\lambda$ of $\,-\nh\hat\Delta\,$ satisfying condition (iii), are curves
of metrics on our closed surface $\,M\nh$, emanating from the fixed metric 
$\,\hg\,$ of nonzero constant Gauss\-i\-an curvature $\,\hat K\nh$.
\begin{lemma}\label{rlizs}Every metric\/ $\,g\ne\hg\,$ in any\/
$\,\lambda$-branch, close to\/ $\,\hg\nh$, realizes case\/ {\rm(b)} of 
Theorem\/~{\rm\ref{srfbs}} and, in particular, has nonconstant Gauss\-i\-an
curvature\/ $\,K\nh$.
\end{lemma}
\begin{proof}By (\ref{int}) -- (\ref{rfr}) the bifurcating branch is a subset
of $\,\xq\nnh^{-\nh1}\nh(0)$, so that (\ref{ltx}) gives (\ref{ktr}) whenever 
$\,g=e^{2x}\hg\,$ for any $\,(x,t)\,$ from the bifurcating branch, with 
$\,a,\gz,c,f\,$ as in (\ref{ltx}). Theorem~\ref{rephr}, the lines preceding 
it, and (\ref{css}) will now yield case (b) of Theorem~\ref{srfbs}, once 
$\,K\hs$ (or, equivalently, $\,f$) is shown to be nonconstant, which we do in
the next paragraph; we cannot simply invoke (\ref{css}), with our fixed
$\,r\ne0$, since Theorem~\ref{rephr} {\it assumes\/} nonconstancy of $\,f\nh$.

The curve of trivial solutions -- see (\ref{rfr}) --- is
contained in $\,\{0\}\times\bbR$, due to the final clause of 
Lemma~\ref{vspec}. It intersects the bifurcating branch transversally, at 
$\,(0,0)\in\mathcal{V}\nh=\mathcal{V}'\nnh\times\bbR$, while both curves lie
on the surface $\,\mathcal{N}\nh\subseteq\mathcal{V}\nh$. (Cf.\ (ii), (a) in
Section~\ref{ov} and (\ref{int}).) A nonzero vector 
$\,(x,t)\,$ tangent to the bifurcating branch at $\,(0,0)\,$ thus has
$\,x\ne0\,$ (or else it would be also tangent to the
triv\-i\-al-so\-lu\-tions curve), and the image of $\,(x,t)\,$ under the
differential $\,d\hh\varPsi\nnh_{(0,0)}\w$, at $\,(0,0)$, of the mapping 
$\,\varPsi\,$ sending $\,(x,t)\in\mathcal{N}\hs$ to the Gauss\-i\-an curvature 
$\,K\,$ of the metric $\,g=e^{2x}\hg\,$ is, from the last line of (\ref{dif}), 
equal to $\,-\hn(\hat\Delta+\nh2\hat K)x$. in view of Lemmas~\ref{vspec}
and~\ref{cndae}, $\,(x,t)\in T\hskip-2.7pt_{(0,0)}\w\hs\mathcal{N}\nh
=\hs\mathrm{Ker}\,d\xq^0_0\times\bbR\,$ and 
$\,\mathrm{Ker}\,d\xq^0_0
\subseteq\mathrm{Ker}\hskip1.7pt(\hat\Delta+\lambda)$, so that  
$\,d\hh\varPsi\nh_{(0,0)}\w(x,t)=-\hn(\hat\Delta+\nh2\hat K)x
=(\lambda-\nh2\hat K)x\,$ is nonzero as 
$\,\lambda\ne\nh2\hat K\hs$ by (iii), and hence also nonconstant, being an 
eigen\-func\-tion of $\,-\nh\hat\Delta\,$ for the positive eigen\-val\-ue 
$\,\lambda$. This, combined with constancy of $\,\hat K\nnh=\varPsi(0,0)$,
implies nonconstancy of $\,\varPsi(x,t)\,$ for all $\,(x,t)\ne(0,0)\,$ in the
bifurcating branch, sufficiently close to $\,(0,0)$.
\end{proof}
The har\-mon\-ic-cur\-va\-ture property of the metric
$\,f\hh^{4/p}[\hs g+\h\hs]\,$ in Theorem~\ref{rephr} is obviously 
unaffected when one multiplies $\,g\,$ and $\,\h\,$ by the same positive 
constant, or separately rescales $\,f\nh$. Our approach removes this freedom, 
by insisting that $\,\gz$ and $\,c\,$ be defined as in (\ref{ltx}): the
metric $\,g=e^{2x}\hg$, for any $\,(x,t)\in\xq\nnh^{-\nh1}\nh(0)$ near
$\,(0,0)$, either equals $\,\hg$, or has nonconstant Gauss\-i\-an curvature,
depending on whether $\,(x,t)\,$ lies in the triv\-i\-al-so\-lu\-tions curve,
or in the bifurcating branch with $\,(0,0)$ removed. Thus, such metrics
include no nontrivial constant multiples of $\,\hg$.

\section{Nonconstant Gauss\-i\-an curvature: conclusion}\label{cn}
\setcounter{equation}{0}
Lemma~\ref{rlizs} implies the second case of (\ref{dch}), that is, (b) in
Theorem~\ref{srfbs}, for $\,M\,$ dif\-feo\-mor\-phic to
$\,S^2\nh,\,\bbR\mathrm{P}^2$ or a closed orientable surface of any genus
$\,\mathbf{g}>1$, and metrics on $\,M\,$ forming nontrivial curves of 
homothety types which, in the case $\,\mathbf{g}>1$, also represent a 
Teich\-m\"ul\-ler-open nonempty set of con\-for\-mal structures.

{\it These metrics give rise to nontrivial compact warped products with
harmonic curvature, having fibres of all dimensions\/} $\,\p\ge2$, {\it and 
any such\/} $\,M\,$ {\it as the base}.

Recall that the $\,\lambda$-branches appearing in Lemma~\ref{rlizs}, for 
eigen\-val\-ues $\,\lambda>0\,$ of $\,-\nh\hat\Delta\,$ satisfying (iii) in
Section~\ref{cb}, constitute curves of metrics on the closed surface 
$\,M\nh$, emanating from the metric $\,\hg\,$ of constant Gauss\-i\-an 
curvature $\,\hat K\nnh\ne0$, and every metric $\,g\,$ near $\,\hg\,$ in the
$\,\lambda$-branch, except $\,\hg\nh$, realizes case (b) of
Theorem~\ref{srfbs}. The metrics in any given $\,\lambda$-branch
\begin{enumerate}
  \def\theenumi{{\rm\alph{enumi}}}
\item represent uncountably many distinct homothety types, and
\item when sufficiently close to $\,\hg\nh$, they cannot be homo\-thet\-ic to
any metric from a $\,\lambda\nh'\nh$-branch, close to $\,\hg\nh$, provided
that $\,\lambda\nh'\hn\ne\lambda$.
\end{enumerate}
First, (a) follows since the homothety invariant
$\,[K\nnh_{\mathrm{max}}\w\nh:\hs K\nnh_{\mathrm{min}}\w]\,$ of
Remark~\ref{hminv}, restricted to any neighborhood of $\,\hg\,$ in the
$\,\lambda$-branch, is nonconstant (and, obviously, continuous): its constancy
would make it equal to $\,[1\hskip-1.9pt:\hskip-2pt1]\,$ (the value of the
invariant for $\,\hg$), and the Gauss\-i\-an curvatures of all the metrics
near $\,\hg\,$ in the $\,\lambda$-branch would thus be constant, contrary to
the final clause of Lemma~\ref{rlizs}.

On the other hand, when a metric $\,g\ne\hg\,$ in the $\,\lambda$-branch
approaches $\,\hg\nh$, the area $\,\mathrm{A}\,$ of $\,(M\nh,g)\,$ tends to
the area $\,\hat{\mathrm{A}}\,$ of $\,(M\nh,\hg)\,$ (clearly equal to 
$\,2\pi/\hat K\hs$ times the Euler characteristic $\,\chi(M)$) and,
consequently, for the homothety invariant $\,\gz\mathrm{A}$ mentioned in the
lines following Theorem~\ref{uniqu}, (\ref{lbd}.ii) implies that
\begin{equation}\label{lim}
\gz\mathrm{A}\,\to\,2\pi\hh(-\nh1+\lambda/\hat K)\hs\chi(M)\hskip7pt\mathrm{as}
\hskip5ptg\to\hg\hskip6pt\mathrm{in\ the\ }\,\lambda\hyp\mathrm{branch.}
\end{equation}
The limits (\ref{lim}) are obviously different for different $\,\lambda$,
which proves (b).

According to Theorem~\ref{srfbs}, for every positive eigen\-val\-ue
$\,\lambda\,$ of $\,-\nh\hat\Delta\,$ having the property (iii-a) or (iii-b) 
of Section~\ref{cb}, the metrics $\,g\ne\hg\,$ in the $\,\lambda$-branch give 
rise to nontrivial compact warped products with harmonic curvature.

In case (iii.a), $\,\lambda\,$ is a positive eigen\-val\-ue of 
$\,-\nh\hat\Delta$, which may be completely arbitrary (if
$\,M=\bbR\mathrm{P}^2$), or of the form $\,\lambda_j\w$ with $\,j\,$ even and
positive (if $\,M=S^2$ and (\ref{eig}) represents the spectrum of
$\,-\nh\hat\Delta$). See Remark~\ref{eignv}.

Condition (iii.b) amounts to requiring that $\,\lambda\,$ be simple and
either greater than $\,(\p\hh-\nh2)\hh|\hat K\hn|\,$ (when $\,r>0$) or
less than $\,(\p\hh-\nh2)\hh|\hat K\hn|\,$ (for $\,r<0$), while
$\,\hat K\nnh<0$.

If $\,r<0\,$ (so that (ii) in Section~\ref{cb} gives $\,\p>2$), or $\,r>0\,$
and $\,\p=2$, the existence of such eigen\-val\-ues $\,\lambda\,$ is
immediate from the result of Schoen, Wolpert and Yau \cite{schoen-wolpert-yau} 
mentioned in Remark~\ref{eighp}.

Finally, when $\,r>0\,$ and $\,\p>2$, we can only provide some anecdotal
evidence for an analogous existence assertion: on the Bolza surface, with the
convention (\ref{eig}), $\,\lambda_{24}\w$ is greater than
$\,23\hh|\hat K\hn|\,$ and simple \cite{strohmaier-uski}; therefore, (iii.b)
holds in this case for all $\,\p\in\{3,4,\dots,21\}$.

To simplify the phrasing of the last two paragraphs, let us unify the two
cases of condition (iii.b) by ignoring the sign of $\,r$. Then, (iii.b) states 
that, on a closed orientable surface of genus $\,\mathbf{g}>1$, with a metric
of constant Gauss\-i\-an curvature $\,\hat K\nnh<0$, the eigen\-val\-ue
$\,\lambda>0\,$ of $\,-\nh\hat\Delta\,$ is simple and different from
$\,(\p\hh-\nh2)\hh|\hat K\hn|$. The result of \cite{schoen-wolpert-yau}
guarantees that, for every genus $\,\mathbf{g}>1$, metrics admitting such 
eigen\-val\-ues $\,\lambda\,$ realize a nonempty open subset of the
Teich\-m\"ul\-ler space.

\section{Constant Gauss\-i\-an curvature: existence\done}\label{cc}
\setcounter{equation}{0}
We now proceed to verify that the first case of (\ref{dch}) -- or,
equivalently, (a) in Theorem~\ref{srfbs} -- holds for a
Teich\-m\"ul\-ler-open, nonempty set of metrics of constant negative
curvatures, on closed orientable surfaces $\,M\,$ of all genera
$\,\mathbf{g}>1$.

{\it This results in nontrivial compact warped products with harmonic
curvature, having fibres of all relevant dimensions\/} $\,\p\ge4$, {\it and
all such\/} $\,M\,$ {\it as the bases}.

The existence assertion needed here is provided by the following result of
Ya\-ma\-be \cite{yamabe}, cf.\ also \cite[pp.\ 115--119]{aubin},
\cite[Lemma 16.37]{besse}, which remains valid even if $\,\dim M=\m>2$, as
long as $\,(\m-2)\hh q<2\m$. The sign of the $\,g\hh$-\nh La\-plac\-i\-an
$\,\Delta\,$ in \cite{besse} is the opposite of ours.
\begin{lemma}\label{ymabe}Given a compact Riemannian surface\/ 
$\,(M\nh,g)$, real numbers\/ $\,q>2$ and\/ $\,c>0$, and\/ $\,a\in\bbR\,$ 
such that\/ $\,(q-2)\hh a>\lambda_1\w$ for the lowest positive eigenvalue\/ 
$\,\lambda_1\w$ of\/ $\,-\nh\Delta$, the equation
\begin{equation}\label{dfm}
\Delta\nh f\hs-\,a\nh f\hs
=\,-\hn cf\hh^{q-1}
\end{equation}
admits a nonconstant positive\/ $\,C^\infty\nnh$ solution\/ $\,f:M\to\bbR$.
\end{lemma}
By (\ref{css}), Lemma~\ref{ymabe} can be applied to case (a) in
Theorem~\ref{srfbs} for compact bases $\,M\nh$. The resulting  construction of
compact warped products with harmonic curvature is a special case of one in
\cite{derdzinski-83} and \cite[Example 16.35(v)]{besse}.

Due to (\ref{css}), equation (\ref{ktr}.iii) then becomes (\ref{dfm}) for
$\,q=2+4/\p\,$ (so that $\,q>2$), while (\ref{ktr}.i) with $\,r=0\,$ reads
$\,\gz=(1\nh-\p)K$. Since the Ein\-stein constant $\,\gz$ of the fibre is
positive (Remark~\ref{parts}), $\,(M\nh,g)\,$ has in this case the negative
constant Gauss\-i\-an curvature $\,K\hs$ and, as 
$\,a=\p(\p\hh-\nh2)\gz/[4(\p\hh-\nh1)]$, the condition 
$\,(q-2)\hh a>\lambda_1\w$ needed in Lemma~\ref{ymabe} is equivalent to
$\,\p>2-\lambda_1\w/K$, cf.\ (\ref{pgt}). However, we are free to assume that
$\,\p\ge4$. (If $\,\p\in\{2,3\}$, the warp\-ed-\hn prod\-uct metric is
con\-for\-mal\-ly flat according to Remark~\ref{cnffl}.) Obviously, having
$\,\p>2-\lambda_1\w/K\hs$ for {\it all\/} $\,\p\ge4$ amounts to the inequality
$\,2>-\lambda_1\w/K$.

According to the second part of Remark~\ref{eighp},
every closed orientable surface of genus greater than $\,1\,$ admits metrics
with negative constant Gauss\-i\-an curvature $\,K\hs$ satisfying this last
inequality, which implies the existence of examples mentioned in the
italicized statement at the beginning of this section.

\section{Constant Gauss\-i\-an curvature: multiplicity\done}\label{cm}
\setcounter{equation}{0}
In equation (\ref{dfm}) we can always assume that $\,c=a$, rewriting it as
\begin{equation}\label{dfc}
\Delta\nh f\hs-\,a\nh f\hs
=\,-\hn af\hh^{q-1},
\end{equation}
since $\,f\,$ may be replaced with $\,(a/c)^{q-2}\nh f\nh$. This normalization
removes the freedom of simultaneously rescaling $\,f\,$ and $\,c$, which is of
no geometric interest.

There are various known multiplicity results for positive solutions of
(\ref{dfc}) on compact Riemannian manifolds $\,(M\nh,g)\,$ of all dimensions 
$\,\m\ge2$. Consider
\begin{equation}\label{nmg}
\begin{array}{l}
\mathrm{the\ number\ }\hh\#(M\nh,g,a,q)\hh\mathrm{\ of\ distinct\
nonconstant}\\
\mathrm{positive\ smooth\ solutions\ }\,f\,\mathrm{\ to\ (\ref{dfc})\ on\
}\,(M\nh,g),
\end{array}
\end{equation}
so that $\,0\le\#(M\nh,g,a,q)\le\infty$. Typically, a lower
bound on $\,\#(M\nh,g,a,q)\,$ is given in terms of the topology of $\,M\nh$.

One defines the {\it Lus\-ter\-nik--Schni\-rel\-mann category\/} $\,\cat(X)\,$
of a topological space $\,X\hs$ to be the least integer $\,k\ge1\,$ such that
$\,X\hs$ is the union of $\,k\,$ open contractible subsets. If no such $\,k\,$
exists, one sets $\,\cat(X)=\infty$. Let us also denote by 
$\,\mathrm{b}_i\w\nh(X,\bbK)$ the $\,i\hh$th Betti number of $\,X\hs$ with 
coefficients in any given field $\,\bbK$, and by $\,\mathrm{b}(X,\bbK)$ the
sum $\,\sum_i\w\mathrm{b}_i\w\nh(X,\bbK)\le\infty$. Thus, $\,\cat(S^\m)=2$,
while any closed surface $\,M\,$ of genus $\,\mathbf{g}\ge1\,$ has
$\,\cat(M)=3\,$ and $\,b(M\nh,\bbZ_2\w)=2(1+\mathbf{g})$.

Finally, one calls a solution $\,f\,$ of (\ref{dfc}) {\it nondegenerate\/} if
it is nondegenerate as a critical point of the associated energy functional
or, equivalently, if the linearized equation 
$\,\Delta\hn\psi-a\psi=a(1-q)f\hh^{q-2}\psi\,$ holds only for the trivial
solution $\,\psi=0$.
\begin{theorem}\label{mrthn}Given a compact Riemannian manifold\/ 
$\,(M\nh,g)\,$ of dimension\/ $\,\m\ge2$, any sufficienly large\/ 
$\,a\in(0,\infty)$, and any\/ $\,q\in(2,2\m/(\m-2))$, with\/ 
$\,2\m/(\m-2)=\infty\,$ if\/ $\,\m=2$, one has\/ $\,\#(M\nh,g,a,q)>\cat(M)$,
in the notation of\/ {\rm(\ref{nmg})}. For any field\/ $\,\bbK\,$ and any
sufficienly large\/ $\,a\,$ such that all nonconstant smooth solutions of\/
{\rm(\ref{dfc})} are nondegenerate,
$\,\#(M\nh,g,a,q)>2\hh\mathrm{b}(M\nh,\bbK)-2$.
\end{theorem}
\begin{proof}
This is the central result of \cite{benci-bonanno-micheletti}, where it is
stated (for reasons not clear to us) only for $\,m\ge3$. However, the proof
remains completely valid also in the case $\,m=2$, due to the So\-bo\-lev
embedding theorem. The same result, with exactly the same proof, also appears
in \cite[Theorem~1.2]{petean}, with no restriction on $\,m\ge2$.
\end{proof}
In the case of hyperbolic surfaces $\,M\nh$, Theorem~\ref{mrthn} with
$\,\cat(M)=3\,$ yields 
\begin{corollary}\label{crlry}On any closed orientable surface of genus
greater than\/ $\,1$, endowed with a metric of negative constant Gauss\-i\-an
curvature, equation\/ {\rm(\ref{dfc})} has at least four distinct nonconstant
positive smooth solutions\/ $\,f\nh$.
\end{corollary}

\bibliographystyle{amsplain}

\end{document}